\documentclass{mcom-l}

\newtheorem{tm}{Theorem}[section]
\newtheorem{ap}{Assumption}[section]

\newtheorem{prop}{Proposition}[section]
\newtheorem{lm}{Lemma}[section]

\theoremstyle{remark}
\newtheorem{rk}{Remark}[section]

\numberwithin{equation}{section}

\usepackage{graphicx} 
\usepackage{extarrows}
\usepackage{latexsym}
\usepackage{amsmath}
\usepackage{amssymb}
\usepackage{amsfonts}
\usepackage{verbatim}
\usepackage{mathrsfs}
\usepackage[modulo]{lineno} 
\usepackage{color}
\usepackage{xcolor}
\usepackage[colorlinks,citecolor=blue,urlcolor=blue]{hyperref}
\usepackage{soul}

\allowdisplaybreaks \allowdisplaybreaks[4]

\begin{document}

\title[Uniform Strong Error Estimates of Tamed-FEM to Superlinear SPDEs]
{Uniform-in-time Strong Error Estimates of Tamed-FEM to Superlinear SPDEs driven by Multiplicative Noise}
 
\author{Jingjing CAI}
\address{Department of Mathematics, Southern University of Science and Technology, Shenzhen, 518055, P.R.China}
\email{12431001@mail.sustech.edu.cn}
 
\author{Zhihui LIU}
\address{Department of Mathematics \& National Center for Applied Mathematics Shenzhen (NCAMS) \& Shenzhen International Center for Mathematics, Southern University of Science and Technology, Shenzhen, 518055, P.R.China}
\email{liuzh3@sustech.edu.cn (Corresponding author)}

\thanks{The authors are supported by the National Natural Science Foundation of China (NNSFC), No. 12671474, Basic and Applied Basic Research Foundation of Guangdong Province, No. 2024A1515012348, and Shenzhen Basic Research Special Project (Natural Science Foundation) Basic Research (General Project), No. JCYJ20240813094919026.}





\begin{abstract}
We establish sharp, uniform-in-time strong error estimates for a nonlinearity-explicit tamed finite element method (FEM) applied to a class of superlinear stochastic partial differential equations (SPDEs) driven by multiplicative noise, including the stochastic Allen--Cahn equation with a moderately thick interface. This tamed-FEM was first introduced in {\it Z. Liu and J. Shen, J. Sci. Comput., 109:Paper No. 61, 26, 2026} to ensure long-time unconditional stability and to preserve the Lyapunov structure of this class of SPDEs. We further prove that the scheme is exponentially ergodic and derive the convergence rate of the numerical invariant measure to the exact one in the Wasserstein-2 distance. Finally, we present numerical experiments that verify the ergodicity, sharpness, and time-independence of the strong convergence rates for this tamed-FEM.
\end{abstract}

\maketitle

\section{Introduction}

There is a well-developed theory of strong error estimates for numerical approximations of SPDEs with Lipschitz coefficients over a finite time horizon; see, e.g., \cite{ACLW16, CHL17, JR15} and references therein. This theory has recently been extended to SPDEs with monotone and superlinear coefficients driven by multiplicative noise in \cite{Liu26, LQ21}. In the past decade, many researchers have analyzed this question for SPDEs with non-Lipschitz coefficients, such as the 2D stochastic Navier--Stokes equation \cite{Dor12}, the 1D stochastic Burgers equation and the Cahn--Hilliard--Cook equation \cite{HJ14}, and the 1D stochastic nonlinear Schr\"odinger equation \cite{CHL17, CHLZ19}, by exploiting the exponential integrability property. In contrast, much less is known about strong error estimates for numerical approximations of SPDEs that remain valid over infinite time horizons.

For superlinear SPDEs, long-time analysis is particularly delicate.
It is well known that the backward Euler method studied in \cite{Liu26} and other implicit schemes for SPDEs often entail high computational costs; nonlinearity-explicit schemes are generally preferred to reduce these costs. However, it was demonstrated in \cite{BHJKLS19} that the classical Euler--Maruyama scheme and its Galerkin-based full discretizations, when applied to a class of superlinear SPDEs, lead to a blow-up of the $p$-th moment for all $p \geq 2$.   
 
To prevent blow-up, \cite{LS26} recently constructed the first nonlinearity-explicit tamed-FEM for a class of superlinear SPDEs driven by multiplicative noise. The authors showed that this tamed-FEM is unconditionally stable over infinite time intervals and uniquely ergodic, and that it strongly converges to the original SPDE at an optimal rate on any finite time interval. However, whether these rates remain valid over an infinite time horizon remains unknown, posing an intriguing and challenging problem.  

This challenge has driven significant research over the past decade, inspiring the development and long-time numerical analysis of the second-order parabolic SPDE
\begin{align}\label{see-fg} 
& {\rm d} X(t, \xi)  
=(\Delta X(t, \xi)+f(X(t, \xi))) {\rm d}t
+g(X(t, \xi)) {\rm d}W(t, \xi),\quad (t, \xi) \in \mathbb R_+\times \mathcal O,
\end{align}
with the homogeneous Dirichlet boundary condition (DBC) $X(t, \xi)=0$, $(t, \xi) \in \mathbb R_+\times \partial \mathcal O$, and initial condition $X(0, \xi)=X_0(\xi)$, $\xi \in \mathcal O$. Here, $\mathcal O \subset \mathbb R^d$ ($d=1,2,3$) is a bounded open set with a piecewise smooth boundary $\partial \mathcal O$; $f:\mathbb R\to\mathbb R$ is assumed to be of monotone type with superlinear growth, with $g:\mathbb R\to\mathbb R$ Lipschitz continuous in the standard infinite-dimensional sense; and $W$ is an infinite-dimensional Wiener process (cf. Section \ref{sec2} for details). Equation \eqref{see-fg} includes the stochastic Allen--Cahn equation, which arises from phase transitions in materials science under stochastic perturbations, corresponding to $f(\xi)=\epsilon^{-2} (\xi-\xi^3)$, $\xi \in \mathbb R$, where $\epsilon>0$ is the interface thickness; see, e.g., \cite{BGJK23, BP24, CH19, FLZ17, HS23, LL26, Liu22, LQ20, LQ21, QW26} and references therein. 

To derive uniform-in-time strong error estimates for the tamed-FEM (see Theorems \ref{tm-err} and \ref{tm-err-}), we establish a uniform-in-time $\dot H^1$-moment bound for the auxiliary process (see Proposition~\ref{lm-reg-aux}), along with a sequence of new key estimates for the tamed function (Lemma \ref{lm-ftau}).
In particular, the uniform upper bound \eqref{ftau'+} for $f_\tau' + \tau |f_\tau'|^2$ plays a crucial role, as it does not require a bound on $|f_\tau'|$ that is uniform in $\tau$. To our knowledge, this is the first such estimate for nonlinearity-explicit schemes that preserve the unique ergodicity of superlinear SPDEs driven by multiplicative noise.
 
We note that the closely related recent work \cite{QW26} established strong and weak error estimates of another tamed-FEM for additive-noise-driven SPDEs with odd-degree polynomial drifts whose leading coefficient is negative. 
Their tamed drift function, which depends on both temporal and spatial discretization parameters, is given by
\begin{align*}
f_{\tau,h}(\xi) := \frac{f(\xi)}{[1 + (\beta_1 \tau^\theta + \beta_2 h^\rho)|\xi|^{\frac{2q-2}{\alpha}}]^\alpha}, \quad \xi \in \mathbb R,
\end{align*}
for \(\alpha \in (0, 1]$ and $\theta, \rho, \beta_1, \beta_2 > 0$ with $\max \{\frac{1}{2}\alpha\rho, \alpha\theta \} < 1 + \frac{d}{2q(2q-1)} - \frac{d}{4}.$  
Their main assumptions required the initial datum $X_0 \in \dot H^\gamma \subset L_\xi^\infty$ for some $\gamma \in (\frac{(q-1)d}{2q-1}, 2]$ and that $\| A^{\frac{\gamma-1}{2}} Q^{\frac{1}{2}} \|_{\mathcal{L}_2(H)} + \| W_A\|_{L_t^\infty L_\omega^p L_\xi^{4q-2}} < \infty$ for all $p \ge 2$, where $A$ denotes the negative Dirichlet Laplacian and $W_A$ denotes the corresponding Ornstein--Uhlenbeck process.
Under the same type of strict dissipativity condition, they derived the uniform-in-time strong convergence rate $h^{\min\{\gamma,\rho\}} + \tau^{\min\{\gamma/2,\theta\}}$, corresponding to $h^{\min\{\gamma,2\}} + \tau^{\min\{\frac{\gamma}{2},1\}}$ by fixing \(\rho = 2, \theta = 1\) and \(\alpha < 1 + \frac{d}{2q(2q-1)} - \frac{d}{4}\); see \cite[Theorem~6.4]{QW26}.  

In contrast, under our sharp regularity requirement $X_0 \in \dot H^{1+\gamma}$ for any $\gamma \in [0, 1]$ (where the weaker \(L^{2(3q+1)}\)-boundedness is used only in the additive-noise case to derive higher-order convergence), along with the same type of dissipativity assumption (with an additional diffusion index due to  the multiplicative noise), the present work permits more general multiplicative noise and drift functions via the taming function
\begin{align} \label{f-tau}
 f_\tau (\xi):=\frac{f(\xi)}{(1+\tau|\xi|^{2q})^{1/2}}, \quad \xi \in \mathbb R, 
\end{align} 
that contains no auxiliary exponents.
Moreover, our strong convergence orders are optimal for any $p \ge 2$:
\begin{align*} 
\sup_{n \in \mathbb N_+} \|X(t_n) - Y_n^h \|_{L_\omega^p L_\xi^2} 
\le 
\begin{cases}
C (h^{1+\gamma}+\tau^{1/2} ) & \quad \text{for multiplicative noise}, \\
C (h^{1+\gamma}+\tau^{(1+\gamma)/2} ) & \quad \text{for additive noise}.
\end{cases}
\end{align*} 

Furthermore, by combining our uniform-in-time strong error estimates with the exponential ergodicity of the tamed-FEM, which follows from the Lyapunov structure derived in \cite{LS26} and an exponential continuity estimate with respect to the initial data, we obtain an ergodic error estimate -- between the exact and numerical invariant measures -- in the Wasserstein-2 distance (see Theorem \ref{tm-erg}).  
We also design numerical experiments to demonstrate the ergodicity, sharpness, and time-independence of the strong convergence rates for this tamed-FEM; to our knowledge, this is the first numerical test of the time-independence of convergence rates for SPDEs.

The paper is organized as follows. In Section \ref{sec2}, we introduce some preliminaries. Several a priori uniform-in-time estimates of Eq. \eqref{see-fg}, the tamed-FEM, and a related auxiliary process are given in Section \ref{sec3}. We establish the strong error estimates between the tamed-FEM and Eq. \eqref{see-fg} in Section \ref{sec4}. In Section \ref{sec5}, we show the exponential ergodicity of the tamed-FEM and an ergodic error estimate in the Wasserstein-2 distance. Finally, we design numerical experiments to verify the ergodicity, as well as the sharpness and time-independence of the strong convergence rates, for the tamed-FEM in Section \ref{sec6}. 

\section{Preliminaries}
\label{sec2}

In this section, we introduce preliminaries, including notations and main assumptions, and then propose the considered tamed-FEM.

\subsection{Notations}

For $p \in [1,\infty]$, we denote by $(L_\xi^p, \|\cdot\|_{L_\xi^p})$, $(L^p_t, \|\cdot\|_{L^p_t})$, $(L_\omega^p, \|\cdot\|_{L_\omega^p})$ the usual real-valued Lebesgue spaces in $\mathcal O$, $\mathbb R_+$, and the sample space $\Omega$, respectively. 
For convenience, we sometimes use the temporal, sample-path, and spatially mixed norm 
$\|\cdot\|_{L^p_\omega L^{p_1}_t L_\xi^{p_2}}$ in different orders.

Let $H := \{u \in L_\xi^2 : u|_{\partial\mathcal O}=0\}$, with norm and inner product denoted by $\|\cdot\|$ (or $\|\cdot\|_{L_\xi^2}$) and $\langle \cdot,\cdot\rangle $, respectively. 
 We denote by $A: {\rm Dom}(A)\subset H\rightarrow H$ the Dirichlet Laplacian on $H$.
Then $-A$ possesses a sequence of positive eigenvalues $\{\lambda_k\}_{k \in \mathbb N_+}$ in an increasing order corresponding to the eigenvectors $\{e_k\}_{k \in \mathbb N_+}$ which vanish on $\partial \mathcal O$:
$-A e_k=\lambda_k e_k$, $k \in \mathbb N_+$.  
Moreover,  $A$ is the infinitesimal generator of an analytic $\mathcal C_0$-semigroup $S(\cdot)=e^{A\cdot}$, and  one can define the fractional powers $(-A)^\theta$ for $\theta\in \mathbb R$ of $-A$.
Let $\dot H^\theta$ be the domain of $(-A)^{\theta/2}$ equipped with the norm $\|\cdot\|_\theta$:
$\|u\|_\theta:=\|(-A)^{\theta/2} u\|$, $u \in \dot H^\theta$.
In particular, one has $\dot H^1=H_0^1 := \{u, \nabla u \in L_\xi^2 : u|_{\partial\mathcal O}=0\}$ whose inner product is given by 
$\langle \cdot, \cdot\rangle _1:=\langle \cdot, \cdot\rangle +\langle \nabla \cdot, \nabla \cdot\rangle .$
It is clear that $\dot H^{-1}$ is the dual space, with the dualization denoted by $\langle \cdot, \cdot\rangle _{-1, 1}$, of $\dot H^1$ (with respect to $\langle \cdot, \cdot\rangle $). 
The following Poincar\'e inequalities will be used throughout:
\begin{align} \label{poin}
\|\nabla u\|^2 \ge \lambda_1 \|u\|^2, \quad u \in \dot H^1; \quad 
\|A v\|^2 \ge \lambda_1 \|\nabla v\|^2, \quad v \in \dot H^2.
\end{align} 
   
Let $U$ be another separable Hilbert space and $Q$ be a self-adjoint and nonnegative definite operator on $U$.
Denote by $U_0:=Q^{1/2} U$ and 
$(\mathcal L_2^\theta:=HS(U_0; \dot H^\theta), \|\cdot\|_{\mathcal L_2^\theta})$ the space of Hilbert--Schmidt operators from $U_0$ to $\dot H^\theta$ for $\theta \in \mathbb R_+$.
Let $W:=\{W(t):\ t\ge0\}$ be a $U$-valued $Q$-Wiener process in the stochastic basis $(\Omega,\mathcal F,(\mathcal F_t)_{t\ge0},\mathbb P)$, i.e., there exists an orthonormal basis $\{g_k\}_{k=1}^\infty$ of $U$ which forms the eigenvectors of $Q$ with the eigenvalues $\{q_k\}_{k=1}^\infty$: $Q g_k= q_k g_k$, $k \in \mathbb N_+$, and a sequence of mutually independent 1D Brownian motions $\{\beta_k\}_{k=1}^\infty $ such that  
$W=\sum_{k \in \mathbb N_+} \sqrt{q_k} g_k \beta_k(t)$. 

Throughout, we use $C$, $c$, $c_1$, etc., to denote generic constants independent of various discrete parameters, which may differ from one appearance to another.

\subsection{Main Assumptions}

Our central assumption on the drift function $f$ in \eqref{see-fg} is the following dissipativity and polynomial growth conditions.

\begin{ap} \label{ap-f}
$f: \mathbb R \to \mathbb R$ is twice continuously differentiable and there exist constants $q \ge 1$, $K_1, K_3 \in \mathbb R$, $K_2>0$, and $L_i>0$, $i=1,\dots,6$ such that for all $\xi \in \mathbb R$ and $\tau \in (0, 1)$,
\begin{align}    
 \xi f(\xi) \le K_1 - K_2 |\xi|^{q+2}, \label{f'} \\
 [1+\tau |\xi|^{2q}]f'(\xi) - q \tau |\xi|^{2(q-1)} \xi f(\xi) 
& \le K_3 (1+\tau|\xi|^{2q})^{3/2}, \label{ftau'}\\
|f(\xi)| \le L_1 + L_2 |\xi|^{q+1}, ~ 
	|f'(\xi)| \le L_3 + L_4 |\xi|^q, ~ 
& |f''(\xi)| \le L_5 + L_6 |\xi|^{q-1}. \label{f-grow}
\end{align}
\end{ap}

Denote by $\tau \in (0, 1)$ the temporal step-size. 
We tame the superlinear drift $f$ by 
 \begin{align} \label{f-tau}
 f_\tau (\xi):=\frac{f(\xi)}{(1+\tau|\xi|^{2q})^{1/2}}, \quad \xi \in \mathbb R.
\end{align}   

\begin{rk}
It is clear that the condition \eqref{ftau'} is equivalent to the uniform upper bound condition $f_\tau'(\xi)=([1+\tau |\xi|^{2q}]f'(\xi) - q \tau |\xi|^{2(q-1)} \xi f(\xi)) / (1+\tau|\xi|^{2q})^{3/2} \le K_3$, and is satisfied for the stochastic Allen--Cahn equation with $K_3=\epsilon^{-2}$: for  $f(\xi)=\epsilon^{-2} (\xi-\xi^3)$, one has 
$f_\tau'(\xi)
=\epsilon^{-2} (1-3\xi^2-\tau\xi^4-\tau\xi^6) (1+\tau\xi^4)^{-3/2}
\le \epsilon^{-2}$, $\xi\in\mathbb R.$
\end{rk}

Throughout, we assume that $q\ge1$ for $d=1,2$ and $q \in [1,2]$ for $d=3$, so that the Sobolev embeddings 
\begin{align}\label{emb}
\dot H^1\hookrightarrow L^{2(q+1)}_\xi \hookrightarrow\; H \hookrightarrow\; L^{2(q+1)/(2q+1)}_\xi \hookrightarrow\; \dot H^{-1},\quad d=1,2,3,
\end{align} 
hold (see \cite{LQ21} for the reason of this range of $q$).
Then we can define the Nemytskii operators $F: \dot H^1 \rightarrow \dot H^{-1}$ and $G: H \rightarrow \mathcal L_2^0$ associated with $f$ and $g$, respectively, by
\begin{align} 
F(u)(\xi):& =  f(u(\xi)),\quad u \in \dot H^1,\ \xi \in \mathcal O, \label{df-F}\\
G(u) g_k(\xi):& =  g(u(\xi)) g_k(\xi), \quad u \in H,~ k \in \mathbb N_+,~ \xi \in \mathcal O.
\label{df-G}
\end{align}

We will frequently use the following facts about $F$ defined in \eqref{df-F} and the corresponding tamed Nemytskii operator $F_\tau$ of $f_\tau$ defined in \eqref{f-tau}; see \cite[Lemma 4.1]{LS26}.

\begin{rk} \label{rk-ftau}
Let \eqref{f'} and \eqref{ftau'} hold. For all $u, v \in \dot H^1$ and $z \in L_\xi^{2(3q+1)}$, 
\begin{align}  
\|F(u)-F(v)\|_{-1}+ \|F_\tau(u)-F_\tau(v)\|_{-1} 
& \le C (1+\|u\|^q_1+\|v\|^q_1 ) \|u-v\|,   \label{Ftau-} \\ 
\|F(u)-F_\tau(u)\|_{-1} 
& \le C \tau^{1/2} (1+\|u\|_1^{2q+1}),   \label{F-Ftau} \\ 
\|F(z)-F_\tau(z)\|
& \le C \tau (1+\|z\|_{L_\xi^{2(3q+1)}}^{3q+1}).  \label{F-Ftau+} 
\end{align} 
\end{rk}

Our main assumption about the diffusion operator $G$ defined in \eqref{df-G} is given by the following Lipschitz continuity and linear growth conditions.
We refer to \cite[Remarks 2.2-2.3 and Example 5.1]{LQ21} for examples for which Assumptions \ref{ap-f}-\ref{ap-g} hold.

\begin{ap} \label{ap-g}
\begin{enumerate}
\item[(1)]
$G: H\rightarrow \mathcal L_2^0$ is Lipschitz continuous and $G: \dot H^1 \rightarrow \mathcal L_2^1$ grows linearly, i.e., there exist positive constants $K_4, K_5$ and $K_6$ such that 
\begin{align}   
\|G(u)-G(v)\|_{\mathcal L_2^0} \le K_4 \|u-v\|, & \quad u,v \in H, \label{g-lip} \\ 
\|G(z) \|_{\mathcal L_2^1}^2 \le K_5 + K_6\|z\|_1^2, & \quad z \in \dot H^1. \label{g1} 
\end{align}

\item[(2)] 
There exists a constant $\theta \in (0, 1)$ such that $G: \dot H^{1+\theta} \rightarrow \mathcal L_2^{1+\theta}$ grows linearly, i.e., there exist positive constants $K_7$ and $K_8$ such that  
\begin{align} \label{g2}
\|G(z) \|_{\mathcal L_2^{1+\theta}}
\le K_7 + K_8\|z\|_{1+\theta},
\quad z \in \dot H^{1+\theta}.
\end{align}
\end{enumerate}
\end{ap}

With these preliminaries, Eq. \eqref{see-fg} is equivalent to the infinite-dimensional stochastic evolution equation
\begin{align} \label{see}
{\rm d} X(t)=(A X(t)+F(X(t))) {\rm d}t+G(X(t)) {\rm d}W, \quad  t \ge 0.
\end{align}

    \subsection{Tamed-FEM}

Let $h\in (0,1)$, $\mathcal T_h$ be a regular family of quasi-uniform partitions of $\mathcal O$ with maximal length $h$, and $V_h \subset \dot H^1$ be the space of continuous functions on $\bar {\mathcal O}$ which are piecewise linear over $\mathcal T_h$ and vanish on $\partial \mathcal O$. 
Let $A_h: V_h \rightarrow V_h$ and $\mathcal P_h: \dot H^{-1} \rightarrow V_h$ be the discrete Dirichlet Laplacian and generalized orthogonal projection operators, respectively, defined by 
\begin{align*}  
\langle A_h u^h, v^h\rangle  & =-\langle \nabla u^h, \nabla v^h\rangle ,
\quad u^h, v^h\in V_h,  \\
\langle \mathcal P_h u, v^h\rangle  & = \langle u, v^h\rangle_{-1, 1},
\quad u\in \dot H^{-1},\ v^h\in V_h. 
\end{align*}

We discretize Eq. \eqref{see} in space with the Galerkin finite element method (FEM) and in time with the nonlinearity-tamed Euler scheme, called the tamed-FEM, i.e., find $V_h$-valued (time-homogeneous) Markov chains $\{Y_n^h, \mathcal F_{t_n}\}_{n \in \mathbb N}$ such that 
 \begin{align} \label{t-fem}
  Y_n^h=Y_{n-1}^h+A_h Y_n^h \tau
  + \mathcal P_h F_\tau (Y_{n-1}^h) \tau 
  + \mathcal P_h G(Y_{n-1}^h)\delta_{n-1} W,   
 \end{align} 
 where $\delta_{n-1} W:=W(t_n) -W(t_{n-1})$, $t_n=n \tau$, $n \in \mathbb N_+$, with initial datum $Y^h_0=\mathcal P_h X_0$.  
The tamed-FEM \eqref{t-fem} is linearly implicit, so it can be uniquely solved pathwise.

The tamed-FEM \eqref{t-fem}  can be equivalently rewritten as  
\begin{align}\label{full}
Y_n^h=S_{h,\tau} Y_{n-1}^h+\tau S_{h,\tau} \mathcal P_h F_\tau(Y_{n-1}^h)
+S_{h,\tau} \mathcal P_h G(Y_{n-1}^h) \delta_{n-1} W,
\quad n \in \mathbb N_+,
\end{align}
where $S_{h,\tau}:=({\rm Id}-\tau A_h)^{-1}$, with ${\rm Id}$ denoting the identity operator in $V_h$, is an approximation of $S$ in one step.
Iterating \eqref{full} yields
\begin{align} \label{full-sum}
Y^h_n 
=S_{h,\tau}^n  Y^h_0+\tau \sum_{i=0}^{n-1}  S_{h,\tau}^{n-i} \mathcal P_h F_\tau(Y^h_i)
+\sum_{i=0}^{n-1}  S_{h,\tau}^{n-i} \mathcal P_h G (Y^h_i) \delta_i W,
\quad n \in \mathbb N_+.
\end{align}

 \section{Uniform-in-time A Priori Estimates}
 \label{sec3}

In this section, we derive several uniform-in-time a priori estimates on the exact solution for Eq. \eqref{see}, the tamed-FEM  \eqref{t-fem}, and its auxiliary process.

We will need the following well-known ultracontractive and smoothing properties of the analytic $\mathcal C_0$-semigroup $S$:
\begin{align}
\|S(t) u\|_\mu \le C e^{-ct} t^{-\frac{\mu-\nu}2} \|u\|_\nu,   
& \quad \forall~ t > 0, ~ 0 \le \nu \le \mu \le 2,~ u \in \dot H^\nu, \label{ana} \\
\|(S(t)-{\rm Id}_H) u\| \le C t^{\frac\rho2} \|u\|_\rho, \label{ana-hol} 
& \quad \forall~ t > 0, ~ 0 \le \rho \le 2, ~ u \in \dot H^\rho.
\end{align}   
Similarly to \eqref{ana}, there holds that  
\begin{align}  \label{sht}
\|S_{h,\tau}^k P_h u\|_\mu \le C t_k^{-\frac{\mu-\nu}2} \Big(1+ \frac12\tau \lambda^* \Big)^{-k} \|u\|_\nu, \quad & \forall~ k \in \mathbb N_+, ~ 0 \le \nu \le \mu \le 2,~ u \in \dot H^\nu,
\end{align}   
where $\lambda^*: = \inf_{h>0} \lambda_1^h$ and $\lambda_1^h$ is the first eigenvalue of the discrete Dirichlet Laplacian $A_h$. 
Due to the quasi-uniformness of the grid and DBC, the minimum discrete eigenvalue has a positive lower bound: $\lambda^*>0$. 
Taking $u=(-A_h)^{-1/2} v$ with $v \in \dot H^{-1}$ and $(\mu, \nu)=(1, 0)$, in combination with the fact that $\|A_h^{-1/2} P_h v\| \le \|v\|_{-1}$, implies 
\begin{align} \label{sht-1} 
\|S_{h,\tau}^k \mathcal P_h v\| \le C t_k^{-1/2} \Big(1+ \frac12\tau \lambda^* \Big)^{-k} \|v\|_{-1},
& \quad \forall~ k \in \mathbb N_+, ~ v \in \dot H^{-1}.
\end{align}

\subsection{Regularity of exact solution}

Let us begin with the following uniform-in-time regularity of the exact solution $X$ to Eq. \eqref{see}.
Here and below, we denote by $\mathcal C_t^\alpha L_\omega^p \dot H^\beta$ the space of adapted stochastic processes $Z$ such that their $\|\cdot\|_{\mathcal C_t^\alpha L_\omega^p \dot H^\beta}$-seminorm is finite:
$\|Z\|_{\mathcal C_t^\alpha L_\omega^p \dot H^\beta}:= \sup_{0\le s < t <\infty} \|Z(t)-Z(s)\|_{L_\omega^p \dot H^\beta} |t-s|^{-\alpha}<\infty$.
For simplicity, we assume that the initial datum $X_0$ is a deterministic element in the rest of the paper.

\begin{lm} \label{lm-u}
\begin{enumerate}
\item
Let $p\ge2$, $\gamma\in [0,1]$, $X_0 \in \dot H^{1+\gamma}$, and Assumptions \ref{ap-f}-\ref{ap-g}(1) (and (2) if $\gamma=1$) hold with $\lambda_1>K_3+\frac{p(q+1)^2-1}{2}K_6$.  
For any $\beta\in [0,1]$, there exists a constant $C$ such that 
\begin{align}   \label{reg-u}
\|X\|_{L_t^\infty L_\omega^p \dot H^{1+\gamma}}
+ \|X\|_{\mathcal C_t^{\min\{1+\gamma-\beta, 1\}/2} L_\omega^p \dot H^\beta}
& \le C (1+\|X_0\|^{(q+1) ^2}_{1+\gamma}). 
\end{align}  

\item
Let $p\ge2$, $X_0 \in \dot H^1 \cap L_\xi^{2(3q+1)}$ and Assumptions \ref{ap-f}-\ref{ap-g}(1) hold with $\lambda_1>K_3+\frac{p(q+1)-1}{2}K_6$. 
There exists a constant $C$ such that  
\begin{align}   \label{reg-u+}
\|X\|_{L_t^\infty L_\omega^p L_\xi^{2(3q+1)}}
& \le  C (1+\|X_0\|_1^{q+1}+\|X_0\|_{L_\xi^{2(3q+1)}}).
\end{align}
\end{enumerate}  
\end{lm}

\begin{proof}
The first uniform-in-time moment estimate in \eqref{reg-u}  with $\gamma=0$ has been shown in \cite[Proposition 2.1]{Liu26}. Then one can apply the arguments in \cite[Proposition 3.1+Theorem 3.3 and Corollary 3.2]{LQ21} and \cite[Lemma 4.2]{LS26} to show the second H\"older estimate in \eqref{reg-u}  and the $L_\xi^{2(3q+1)}$-estimate \eqref{reg-u+}, taking into account the first uniform moment's estimate.
\end{proof}

\subsection{Estimates of tamed-FEM and auxiliary process}

As in \cite{LQ21}, we introduce the auxiliary process 
\begin{align}\label{aux}
{\widehat Y}_n^h 
=S_{h,\tau}^n \mathcal P_h X_0
+\tau \sum_{i=0}^{n-1} S_{h,\tau}^{n-i} \mathcal P_h F_\tau(X(t_i))
+\sum_{i=0}^{n-1} S_{h,\tau}^{n-i} \mathcal P_h G(X(t_i)) \delta_i W, 
\end{align}
for $n \in \mathbb N_+$, with ${\widehat Y}_0^h=\mathcal P_h X_0$.
It is clear that 
\begin{align} \label{aux+} 
{\widehat Y}_n^h
&={\widehat Y}_{n-1}^h+\tau A_h {\widehat Y}_n^h 
+\tau \mathcal P_h F_\tau(X(t_{n-1}))
+\mathcal P_h G(X(t_{n-1})) \delta_{n-1} W, ~ n \in \mathbb N_+.
\end{align}

Let us begin with the following estimates.

\begin{lm} \label{lm-ftau}
Under Assumption \ref{ap-f}, there exist positive constants $\tau_1 < 1$ and $T_1$ such that for any $\tau \in (0, \tau_1)$, $\xi \in \mathbb R$, and $u, v, w \in \dot H^1$,     
\begin{align}
f_\tau'(\xi)+ \tau|f_\tau'(\xi)|^2 & \le K_3+K_3^2\tau,  \label{ftau'+}  \\
\langle F_\tau(u)-F_\tau(v), u-v \rangle_{-1, 1} + \tau \|F_\tau(u)-F_\tau(v)\|^2 
& \le (K_3+K_3^2 \tau) \|u-v\|^2,  \label{tauftau}   \\ 
\langle  w, F_\tau(w)\rangle _1 + \tau \|F_\tau(w)\|_1^2
& \le T_1+(K_3+K_3^2\tau) \|w\|_1^2. \label{dis-mon}
\end{align} 
\end{lm}

\begin{proof}
By \eqref{f-tau} and the conditions \eqref{f'}-\eqref{f-grow}, we have $f_\tau'(\xi) \le K_3$ and 
\begin{align*}
f_\tau'(\xi)
	&= \frac{[1+\tau |\xi|^{2q}]f'(\xi)-q\tau|\xi|^{2(q-1)} \xi f(\xi)}{(1+\tau |\xi|^{2q})^{3/2}}
	\ge -\frac{L_3+L_4|\xi|^q}{\sqrt{1+\tau |\xi|^{2q}}}
	-qK_1^+\frac{\tau |\xi|^{2q-2}}{(1+\tau |\xi|^{2q})^{3/2}}, 
\end{align*}
$\xi \in \mathbb R$, where $K_1^+:=\max\{0, K_1\}$.
Set $y=\sqrt\tau\,|\xi|^q$. Then
\begin{align*}
	\frac{L_3+L_4|\xi|^q}{(1+\tau|\xi|^{2q})^{1/2}} 
	\le (L_3^2+{L_4^2}/\tau)^{1/2}, \quad 
	\frac{\tau|\xi|^{2q-2}}{(1+\tau|\xi|^{2q})^{3/2}} 
	\le d_q\tau^{1/q},
	\end{align*}
with $d_q:=\sup_{y\ge0}\frac{y^{2-2/q}}{(1+y^2)^{3/2}}$.
Hence
\begin{equation}\label{bound-ftau}
    -m_\tau\le f_\tau'(\xi)\le K_3,
    \qquad m_\tau:= (L_3^2+ L_4^2/\tau )^{1/2}+qK_1^+d_q\tau^{1/q}.
\end{equation}

 Set $\phi_\tau(a):=a+\tau a^2$, $a \in \mathbb R$. Since $\phi_\tau$ is convex, the above bound \eqref{bound-ftau} implies
$\phi_\tau(f_\tau'(\xi)) \le \max\{\phi_\tau(-m_\tau),\phi_\tau(K_3)\}.$ 
For any $\tau \in (0, \tau_0:= \min\{1, 4(L_4+\sqrt{L_4^2+4C_0})^{-2}\})$ with $C_0:=\max\{0,L_3+qK_1^+d_q-K_3\}$, we have
\begin{align*}
	m_\tau\le L_3+ L_4 \tau^{-1/2}+qK_1^+d_q
	\le K_3+C_0+L_4 \tau^{-1/2}
	\le K_3+ \tau^{-1},
\end{align*}
and thus $\phi_\tau(-m_\tau)-\phi_\tau(K_3)=(m_\tau+K_3)\bigl(\tau(m_\tau-K_3)-1\bigr)\le0$.
Consequently,
\begin{align*}
	f_\tau'(\xi)+\tau |f_\tau'(\xi)|^2
	=\phi_\tau(f_\tau'(\xi))
	\le \phi_\tau(K_3)
	=K_3+\tau K_3^2,
\end{align*}
thereby proving \eqref{ftau'+}.

To show \eqref{tauftau} for any $u, v \in \dot H^1$, we observe that $f_\tau(u)-f_\tau(v)=M_\tau(u, v)(u-v)$ with $M_\tau(u, v):=\int_0^1 f_\tau'(v+s(u-v)) ds$.
Then
\begin{align*}
& \langle F_\tau(u)-F_\tau(v), u-v \rangle_{-1, 1} + \tau \|F_\tau(u)-F_\tau(v)\|^2 \\
& = \int_{\mathcal{O}} [ (f_\tau(u) - f_\tau(v))(u-v) + \tau |f_\tau(u) - f_\tau(v)|^2] \, d\xi \\
& = \int_{\mathcal{O}} [M_\tau(u, v) + \tau |M_\tau(u, v)|^2 ] |u-v|^2 \, d\xi.
\end{align*}
Here and in what follows, we omit the integration variable when integration is present to lighten the notation.
By Cauchy--Schwarz inequality, we obtain
\begin{align*}
& \langle F_\tau(u)-F_\tau(v), u-v \rangle_{-1, 1} + \tau \|F_\tau(u)-F_\tau(v)\|^2 \\
& \leq \int_{\mathcal{O}} \Big( \int_0^1 [f_\tau'(v+s(u-v)) + \tau |f_\tau'(v+s(u-v))|^2] \, ds \Big) |u - v|^2 \, d\xi.
\end{align*} 
In combination with the above inequality and \eqref{ftau'+}, we conclude \eqref{tauftau}.

Let $u \in \dot H^1$.
The definition of the inner product and norm in $\dot H^1$ yields that  
\begin{align*}
	&\langle u,F_\tau(u)\rangle _1+\tau\|F_\tau(u)\|_1^2
	 = \langle u,F_\tau(u)\rangle +\tau\|F_\tau(u)\|^2
	+ \int_{\mathcal O} (f_\tau'(u)+ \tau|f_\tau'(u)|^2) |\nabla u|^2 \mathrm{d}\xi.
\end{align*}  
For the first term, following the technique used in the proof of \cite[Lemma 3.1]{LS26}, we have the existence of positive constants $\tau_0 < 1$, $T_1$, and $T_1'$ such that for any $\tau \in (0, \tau_0')$, 
$\langle u,F_\tau(u)\rangle +\tau\|F_\tau(u)\|^2 \le T_1 - T_1' \|u\|^2$.
For the second term, the estimate \eqref{ftau'+} implies  
$\int_{\mathcal O} (f_\tau'(u)+ \tau|f_\tau'(u)|^2) |\nabla u|^2 \mathrm{d}\xi \le (K_3+K_3^2\tau)\|\nabla u\|^2$.
Combining the above two estimates, we conclude \eqref{dis-mon} with $\tau_1:=\min\{\tau_0, \tau_0'\}$.
\end{proof}

Now we can show the following required uniform-in-time estimates on the tamed-FEM  \eqref{t-fem} and its auxiliary process \eqref{aux}.

\begin{prop} \label{lm-reg-aux} 
Let $p \ge 2$, $X_0 \in \dot H^1$, and Assumptions \ref{ap-f}+\ref{ap-g}(1) hold with $\lambda_1 > K_3 + \frac{p(q+1)-1}{2} K_6$.  
There exist positive constants $C$ and $\tau_2 <1$ such that for any $h \in (0,1)$ and $\tau \in (0,\tau_2)$,   
\begin{align}  \label{reg-aux} 
  \sup_{n \in \mathbb N_+} \mathbb E \|{\widehat Y}_n^h\|_1^p  
& \le C (1+\|X_0\|_1^{p(q+1)^2}).
\end{align} 
\end{prop}

\begin{proof}  
Minkowski inequality yields that 
\begin{align*}
\|{\widehat Y}_n^h\|_{L_\omega^{p} \dot H^1} 
& \le \|S_{h,\tau}^n \mathcal P_h X_0\|_{L_\omega^{p} \dot H^1}
+ \tau \sum_{i=0}^{n-1} \|S_{h,\tau}^{n-i} \mathcal P_h F_\tau(X(t_i))\|_{L_\omega^{p} \dot H^1} \\
& \quad + \Big\|\sum_{i=0}^{n-1} S_{h,\tau}^{n-i} \mathcal P_h G(X(t_i)) \delta_i W \Big\|_{L_\omega^{p} \dot H^1}.
\end{align*}
By \eqref{sht}, \eqref{f-grow}, and the embedding \eqref{emb}, we get a bound for the first two terms:  
\begin{align*}
& \|S_{h,\tau}^n \mathcal P_h X_0\|_{L_\omega^{p} \dot H^1}
+ \tau \sum_{i=0}^{n-1} \|S_{h,\tau}^{n-i} \mathcal P_h F_\tau(X(t_i))\|_{L_\omega^{p} \dot H^1} \\
& \le C \|X_0\|_{L_\omega^{p} \dot H^1}
+C \tau \sum_{i=0}^{n-1} t_{n-i}^{-1/2}e^{-c t_{n-i}}  \|F_\tau(X(t_i))\|_{L_\omega^{p} L_\xi^2} \\
& \le C \|X_0\|_{L_\omega^{p} \dot H^1}
+ C (1 + \|X\|^{q+1}_{L_t^\infty L_\omega^{p(q+1)} \dot H^1} ) \Big(\tau  \sum_{k=1}^\infty t_k^{-1/2}e^{-c t_k}\Big).
\end{align*} 

For the last term, discrete BDG and Minkowski inequalities and \eqref{sht} imply
\begin{align*}
& \sup_{1 \le n \le N} \mathbb E \Big\|\sum_{i=0}^{n-1} S_{h,\tau}^{n-i} \mathcal P_h G(X(t_i)) \delta_i W \Big\|_1^p \\ 
& \le C \sup_{1 \le n \le N} \Big( \sum_{i=0}^{n-1} e^{-2 c t_{n-i}} 
\|G(X(t_i)) \delta_i W\|_{L_\omega^{p} \dot H^1}^2 \Big)^{p/2}.
\end{align*}
The continuous BDG inequality and the growth condition \eqref{g1} yield that 
\begin{align*}
\mathbb E  \|G(X(t_i)) \delta_i W\|_1^p
\le C \tau^{p/2} \mathbb E  \|G(X(t_i))\|_{\mathcal L_2^1}^p
\le C \tau^{p/2} (1+ \mathbb E  \|X(t_i)\|_1^p),
\end{align*} 
and thus 
\begin{align*}
& \sup_{1 \le n \le N} \mathbb E \Big\|\sum_{i=0}^{n-1} S_{h,\tau}^{n-i} \mathcal P_h G(X(t_i)) \delta_i W \Big\|_1^p \le C (1 + \|X\|_{L_t^\infty L_\omega^p \dot H^1}^p ) 
\Big( \tau \sum_{k=1}^\infty e^{-2 c t_k} \Big)^{p/2}.
\end{align*}
Combining the above estimates, we obtain 
\begin{align*}
\sup_{0 \le n \le N} \|{\widehat Y}_n^h\|_{L_\omega^{p} \dot H^1} 
& \le C \|X_0\|_{L_\omega^{p} \dot H^1}
+ C (1 + \|X\|^{q+1}_{L_t^\infty L_\omega^{p(q+1)} \dot H^1} ) 
\Big( \tau\sum_{k=1}^\infty t_k^{-1/2}e^{-c t_k} \Big) \\
& \quad + C (1 + \|X\|_{L_t^\infty L_\omega^p \dot H^1} ) 
\Big( \tau \sum_{k=1}^\infty e^{-2 c t_k} \Big)^{1/2}.
\end{align*} 
Using the facts that 
$\sum_{k=1}^\infty k^{-1/2}e^{-c \tau k}
 \le C \tau^{-1/2}$ and 
$\tau \sum_{k=1}^\infty e^{- c \tau k}  
 \le C $,  
we have 
\begin{align*}
\sup_{n \in \mathbb N} \|{\widehat Y}_n^h\|_{L_\omega^{p} \dot H^1} 
& \le C \|X_0\|_{L_\omega^{p} \dot H^1}
+ C (1 + \|X\|^{q+1}_{L_t^\infty L_\omega^{p(q+1)} \dot H^1}). 
\end{align*} 
This, in combination with the uniform regularity \eqref{reg-u}, gives \eqref{reg-aux}.     
\end{proof}

\section{Uniform Strong Error Estimates}
\label{sec4}

In this section, we establish optimal uniform-in-time strong error estimates for the tamed-FEM approximation \eqref{t-fem} of the exact solution of \eqref{see}.

\subsection{Error estimates of auxiliary process}

We will use the following estimates of $E_{h,\tau}(t):=S(t)-S_{h,\tau}^n  \mathcal P_h$ for $t\in (t_{n-1}, t_n ]$ with $n \in \mathbb N_+$, which can be derived similarly to \cite[Lemma 4.1]{LQ21}.

\begin{lm} \label{lm-eht}
Let $t>0$, $0 \le \nu \le \mu \le 2$, and $0 \le \theta \le 1$.
Then
\begin{align} 
\|E_{h,\tau} (t) x\| 
\le C (h^\mu+\tau^\frac\mu2) e^{-ct} t^{-\frac{\mu-\nu}2} \|x\|_\nu,
& \quad \forall~x \in \dot H^\nu, \label{eht1} \\
\int_0^\infty \|E_{h,\tau}(r) y\|^2 {\rm d}r  
\le C (h^{1+\theta}+\tau^\frac{1+\theta}2)^2 \|y\|_\theta^2,
& \quad \forall~ y \in \dot H^\theta. \label{eht2}
\end{align} 
\end{lm}

We first consider the strong error estimate between Eq. \eqref{see} and the auxiliary process \eqref{aux}.
All results hold for $\mathcal F_0$-measurable $X_0$ with certain bounded high-order moments.

\begin{lm} \label{lm-aux} 
Let $p \ge 2$, $X_0 \in \dot H^{1+\gamma}$ with $\gamma\in [0,1]$, and Assumptions \ref{ap-f}-\ref{ap-g}(1) (and (2) if $\gamma=1$) hold with $\lambda_1>K_3+\frac{p(3q+1)(q+1)-1}{2}K_6$. 
There exists a constant $C$ such that    
\begin{align} \label{err-aux}
\sup_{n \in \mathbb N_+} \|X(t_n)-{\widehat Y}_n^h\|_{L_\omega^p L_\xi^2} 
\le C (1 + \|X_0\|_{1+\gamma}^{(q+1)^2}) (h^{1+\gamma}+\tau^{1/2} ).
\end{align}   
\end{lm}

\begin{proof}
Let $n \in \mathbb N_+$.
Subtracting ${\widehat Y}_n^h$ in \eqref{aux} from the mild formulation
$X(t) =S(t) X_0+\int_0^t S(t-r) F(X(r)) {\rm d}r
+\int_0^t S(t-r) G(X(r)) {\rm d}W(r)$
with $t=t_n$, we get
\begin{align} \label{j}
J^n  : = & \|X(t_n )-{\widehat Y}_n^h \|_{L_\omega^p L_\xi^2}
\le \| E_{h,\tau}(t_n ) X_0\|_{L_\omega^p L_\xi^2} \nonumber \\
& + \Big\|\sum_{i=0}^{n-1} \int_{t_i}^{t_{i+1}} [S(t_n-r) F(X(r))-S_{h,\tau}^{n-i} \mathcal P_h F_\tau(X(t_i))] {\rm d}r \Big\|_{L_\omega^p L_\xi^2}\nonumber \\
& + \Big\|\sum_{i=0}^{n-1} \int_{t_i}^{t_{i+1}} [S(t_n-r) G(X(r))-S_{h,\tau}^{n-i} \mathcal P_h G(X(t_i))] {\rm d}W(r) \Big\|_{L_\omega^p L_\xi^2}
=: \sum_{i=1}^3 J^n  _i.
\end{align} 
In the sequel, we treat the above three terms one by one.

The estimation \eqref{eht1} with $\mu=\nu=1+\gamma$ yields that 
\begin{align}  \label{j1}
J^n_1
\le C e^{-c t_n} (h^{1+\gamma}+\tau^\frac{1+\gamma}2 ) \|X_0\|_{1+\gamma},
\quad \gamma\in [0,1].
\end{align}
To deal with the second term, we decompose it into the following three terms:
\begin{align*} 
J^n  _2 
&\le \sum_{i=0}^{n-1} \int_{t_i}^{t_{i+1}} 
\|S(t_n-r) [F(X(r))-F(X(t_i))] \|_{L_\omega^p L_\xi^2} {\rm d}r  \\
&\quad + \sum_{i=0}^{n-1} \int_{t_i}^{t_{i+1}} 
\|E_{h,\tau}(t_n-r) F(X(t_i)) \|_{L_\omega^p L_\xi^2} {\rm d}r \\
&\quad + \sum_{i=0}^{n-1} \int_{t_i}^{t_{i+1}} 
\|S_{h,\tau}^{n-i} \mathcal P_h [F(X(t_i))-F_\tau(X(t_i))]\|_{L_\omega^p L_\xi^2} {\rm d}r 
=: \sum_{j=1}^3 J^n_{2j}.
\end{align*} 
The bound \eqref{ana} with $(\mu, \nu)=(1, 0)$ and the dual estimation \eqref{Ftau-} yield
\begin{align*} 
J^n  _{21} 
&\le C \sum_{i=0}^{n-1} \int_{t_i}^{t_{i+1}} (t_n-r)^{-1/2} e^{-C (t_n-r)}
\|F(X(r))-F(X(t_i)) \|_{L_\omega^p \dot H^{-1}} {\rm d}r \\ 
& \le C \tau^{1/2} (1+\|X\|^q_{L_t^\infty L_\omega^{2pq} \dot H^1} )
\|X\|_{\mathcal C_t^{1/2} L_\omega^{2p} L_\xi^2}, \quad \gamma\in [0,1],
\end{align*}  
where the elementary estimate $\int_0^\infty r^{-\frac{1+\gamma}2} e^{-c r} {\rm d}r \le C<\infty$ for all $\gamma\in [0,1)$.
Similarly, using \eqref{eht1} with $\mu=1+\gamma$ and $\nu=0$ and the embedding \eqref{emb}, we derive  
\begin{align} \label{j22}
J^n_{22}
&\le C (h^{1+\gamma}+\tau^\frac{1+\gamma}2) 
(1+\|X\|_{L_t^\infty L_\omega^p \dot H^1}^{q+1}),
\quad \gamma\in [0,1).
\end{align}  
By \eqref{sht-1} with $\mu=\nu=0$, \eqref{F-Ftau}, and the elementary estimate $\sum_{i=1}^\infty t_i^{-1/2} (1+ c \tau)^{-i}  \le C \tau^{-1}$, we get  
\begin{align} \label{j23} 
J^n  _{23} 
& \le \sum_{i=0}^{n-1} \int_{t_i}^{t_{i+1}} t_{n-i}^{-1/2} (1 + \frac12\tau \lambda^*)^{-(n-i)}  
\|F(X(t_i))-F_\tau(X(t_i))\|_{L_\omega^p \dot H^{-1}} {\rm d}r \nonumber  \\
&\le C \tau^{1/2} (1+\|X\|_{L_t^\infty L_\omega^{p(2q+1)} \dot H^1}^{2q+1}), \quad \gamma\in [0,1].
\end{align}  
Combining the above three estimations and Lemma \ref{lm-u} implies  
\begin{align}  \label{j2}
J^n_2
& \le C (h^{1+\gamma}+\tau^{1/2}) [(1+\|X\|^q_{L_t^\infty L_\omega^{2pq} \dot H^1} ) \|X\|_{\mathcal C_t^{1/2} L_\omega^{2p} L_\xi^2} \nonumber \\ & \quad +(1+\|X\|_{L_t^\infty L_\omega^p \dot H^1}^{q+1})+ \|X\|_{L_t^\infty L_\omega^{p(2q+1)} \dot H^1}^{2q+1}]  \nonumber \\ 
& \le C (h^{1+\gamma}+\tau^{1/2} )  
(1+ \|X_0\|_1^{2q+1}),
\quad \gamma\in [0,1).
\end{align}

For the last term $J^n_3$, it can be controlled by using both the continuous and discrete BDG inequalities (see, e.g., \cite[(2.22) and (2.23)]{LQ21}) as  
\begin{align*} 
(J^n_3)^2  
& \le C \sum_{i=0}^{n-1} \int_{t_i}^{t_{i+1}} 
\|S_{h,\tau}^{n-i} \mathcal P_h [G(X(r))-G(X(t_i))]\|_{L_\omega^p \mathcal L_2^0}^2 {\rm d}r  \\
& \quad + C \sum_{i=0}^{n-1} \int_{t_i}^{t_{i+1}} 
\|E_{h,\tau}(t_n-r) G(X(r)) \|_{L_\omega^p \mathcal L_2^0}^2 {\rm d}r
:=\sum_{j=1}^2 (J^n_{3j})^2.
\end{align*}
Using \eqref{sht}, \eqref{g-lip}, \eqref{g1}, and \eqref{eht1} with $(\mu,\nu)=(1+\gamma,1)$ for $\gamma\in [0,1)$, we get 
\begin{align*} 
(J^n_3)^2
&\le C \|X\|_{\mathcal C_t^{1/2} L_\omega^{2p} L_\xi^2}^2
\times \sum_{i=0}^{n-1}  \int_{t_i}^{t_{i+1}} (r-t_i) e^{-C(n-i) \tau} {\rm d}r \nonumber  \\ 
&\quad +C (h^{1+\gamma}+\tau^\frac{1+\gamma}2)^2 
(1+\|X\|_{L_t^\infty L_\omega^p \dot H^1})^2 
\times  \int_0^{t_n } r^{-\gamma}  e^{-C r} {\rm d}r.
\end{align*} 
It follows from the estimation \eqref{reg-u} that 
\begin{align}  \label{j3} 
J^n_3
& \le C (1+\|X_0\|_1^{q+1} ) (h^{1+\gamma}+\tau^{1/2}), 
\quad \gamma\in [0,1).
\end{align} 
Putting the estimations \eqref{j1}, \eqref{j2}, and \eqref{j3} together results in
\begin{align*} 
J^n  \le C (1+ \|X_0\|_1^{2q+1} )
(h^{1+\gamma}+\tau^{1/2}),  \quad \gamma\in [0,1).
\end{align*}
Combining this inequality with the estimation  
$\|X_0-\mathcal P_h X_0\|_{L_\omega^p L_\xi^2}
\le C h^{1+\gamma} \|X_0\|_{1+\gamma}$,
$\gamma\in [0,1]$, completes the proof of \eqref{err-aux} with $\gamma\in [0,1)$. 

To show \eqref{err-aux} for $\gamma=1$, we just need to give refined estimations for the $J^n_{22}$ and $J^n_{32}$ when $X_0 \in \dot H^2$ and \eqref{g2} holds.
Applying Minkovski inequality and using \eqref{eht1} with $\mu=2$ and $\nu=\beta\in (0,1/2)$, the embedding $\dot H^2 \hookrightarrow \dot H^\beta \cap L_\xi^\infty$, the estimate \eqref{reg-u}, and the fact 
$\|F(u)\|_\beta \le C (1+\|u\|^{q+1}_{L_\xi^\infty} +\|u\|^{q+1}_\beta)$ for all
$u\in \dot H^\beta\cap L_\xi^\infty$ (see \cite[Lemma 3.2]{LQ21}),
we derive  
\begin{align} \label{j22+}
J^n_{22}
&\le C (1+\|X_0\|_2^{(q+1)^2}) (h^2+\tau).
\end{align}
On the other hand, \eqref{eht1} with $\mu=2$ and $\nu=1+\theta$, \eqref{g2}, and \eqref{reg-u} imply that 
\begin{align*} 
J^n_{32}
& \le C (1+\|X_0\|^{q+1}_2 ) (h^2+\tau).
\end{align*}
This completes the proof.
\end{proof}

\subsection{Error estimate in multiplicative noise case}

Combining Lemma \ref{lm-aux} with a variational approach for monotone SPDEs, we have the following uniform-in-time strong convergence rate between the solution $X$ of  Eq. \eqref{see} and the solution $\{Y_n^h\}$ of the tamed-FEM \eqref{t-fem}.

\begin{tm}  \label{tm-err}
Let $p \ge 1$, $X_0 \in \dot H^{1+\gamma}$ with $\gamma\in [0,1]$, Assumptions \ref{ap-f}-\ref{ap-g}(1) (and (2) if $\gamma=1$) hold with $\lambda_1>K_3+\frac12\max\{ (2p-1)K_4^2, (2p(3q+1)(q+1)-1)K_6\}$. 
There exist positive constants $C$ and $\tau_3<1$ such that for any $h \in (0,1)$ and $\tau \in (0,\tau_3)$,
\begin{align} \label{err} 
& \sup_{n \in \mathbb N_+} \|X(t_n) - Y_n^h \|_{L_\omega^{2p} L_\xi^2} 
 \le C (1 +\|X_0\|_{1+\gamma}^{(q+1)(4q+1)}) (h^{1+\gamma}+\tau^{1/2} ).
\end{align} 
\end{tm}

\begin{proof}
Let $n \in \mathbb N_+$ and denote  
$e_n^h:={\widehat Y}_n^h -Y_n^h$. 
Then $e_n^h  \in V_h$ with vanishing initial datum $e^h_0=0$.
In terms of \eqref{aux+} and \eqref{t-fem}, it is clear that  
\begin{align*}
& e_n^h - e_{n-1}^h - \tau A_h e_n^h \\
& = \tau \mathcal P_h [F_\tau(X(t_{n-1}))-F_\tau(Y_{n-1}^h)]  
+\mathcal P_h [G(X(t_{n-1}))-G(Y_{n-1}^h)] \delta_{n-1} W.
\end{align*}
Testing with $e_n^h$ on both sides of the above equation and the elementary equality
\begin{align}\label{ab}
	2\langle a-b,a\rangle =|a|^2-|b|^2+|a-b|^2, 
\end{align}
for any $a,b$ in an arbitrary Hilbert space, we obtain 
\begin{align*}
& \|e_n^h\|^2-\|e_{n-1}^h\|^2 + \|e_n^h -e_{n-1}^h\|^2 +2 \tau \|\nabla e_n^h  \|^2\\
& =  2\tau \langle F_\tau(X(t_{n-1}))-F_\tau({\widehat Y}_{n-1}^h), e_n^h\rangle _{-1, 1}
+ 2\tau \langle  F_\tau({\widehat Y}_{n-1}^h)-F_\tau(Y_{n-1}^h), e_{n-1}^h\rangle_{-1, 1} \\
 & + 2\langle e_n^h  -e_{n-1}^h, \tau [F_\tau({\widehat Y}_{n-1}^h)-F_\tau(Y_{n-1}^h)]+[G(X(t_{n-1}))-G(Y_{n-1}^h)]\delta_{n-1} W\rangle \\
 & + 2\langle e_{n-1}^h , (G(X(t_{n-1}) )-G(Y_{n-1}^h))\delta_{n-1} W\rangle .
\end{align*} 
By the Poincar\'e inequality \eqref{poin} and Young inequality, for any small $\varepsilon>0$, we have 
\begin{align}\label{en+}
	& [1+2(\lambda_1-\varepsilon)\tau]\|e_n^h\|^2
	\le \|R_{n-1}\|^2
	+ C \tau \|F_\tau(X(t_{n-1}))-F_\tau({\widehat Y}_{n-1}^h)\|_{-1}^2,
\end{align} 
with $R_{n-1}:=e_{n-1}^h+\tau \mathcal P_h [F_\tau({\widehat Y}_{n-1}^h)-F_\tau(Y_{n-1}^h)] + \mathcal P_h [G(X(t_{n-1}))-G(Y_{n-1}^h)]\delta_{n-1} W$.

For any integer $p \ge 1$, taking the $p$-th power and $\mathbb E_{n-1}$, we obtain
\begin{align*}
	& \mathbb E_{n-1}\|e^h_n\|^{2p}
	\le \frac{(1+\varepsilon\tau)}{[1+ 2 (\lambda_1-\varepsilon) \tau]^p} \mathbb E_{n-1}\|R_{n-1}\|^{2p}
	+ C \tau \|F_\tau(X(t_{n-1}))-F_\tau({\widehat Y}_{n-1}^h)\|_{-1}^{2p}.
\end{align*}
Set
$Z_{n-1}(t):=e_{n-1}^h+\tau \mathcal P_h[F_\tau({\widehat Y}_{n-1}^h)-F_\tau(Y_{n-1}^h)]+\mathcal P_h[G(X(t_{n-1}))-G(Y_{n-1}^h)](W(t_{n-1}\\+t)-W(t_{n-1}))$ for $0\le t\le \tau$. Then $Z_{n-1}(0)=e_{n-1}^h+\tau \mathcal P_h[F_\tau({\widehat Y}_{n-1}^h)-F_\tau(Y_{n-1}^h)]$ and $Z_{n-1}(\tau)=R_{n-1}$.
Applying It\^o formula to $\|Z_{n-1}(t)\|^{2p}$, followed by taking $\mathbb E_{n-1}$, and using $\|[G(X(t_{n-1}))-G(Y_{n-1}^h)]^*z\|_{U_0}^2\le \|G(X(t_{n-1}))-G(Y_{n-1}^h)\|_{\mathcal L_2^0}^2\|z\|^2$, $z\in H$, we infer
\begin{align*}
\frac{\mathrm{d}}{\mathrm{d} t}\mathbb E_{n-1}\|Z_{n-1}(t)\|^{2p}
&\le p(2p-1)\|G(X(t_{n-1}))-G(Y_{n-1}^h)\|_{\mathcal L_2^0}^2\mathbb E_{n-1}\|Z_{n-1}(t)\|^{2p-2}.
\end{align*}
By H\"older inequality,
$\mathbb E_{n-1}\|Z_{n-1}(t)\|^{2p-2}\le (\mathbb E_{n-1}\|Z_{n-1}(t)\|^{2p})^{\frac{p-1}{p}}$. Substituting this into the above estimate yields
$\frac{\mathrm{d}}{\mathrm{d} t}(\mathbb E_{n-1}\|Z_{n-1}(t)\|^{2p})^{1/p}\le (2p-1)\|G(X(t_{n-1}))-G(Y_{n-1}^h)\|_{\mathcal L_2^0}^2$.
Integrating over $[0,\tau]$ gives
\begin{align}\label{e-n}
\mathbb E_{n-1}\|R_{n-1}\|^{2p}
&\le (\|e_{n-1}^h+\tau [F_\tau({\widehat Y}_{n-1}^h)-F_\tau(Y_{n-1}^h)]\|^2 \notag\\
&\quad +(2p-1)\tau\|G(X(t_{n-1}))-G(Y_{n-1}^h)\|_{\mathcal L_2^0}^2)^p.
\end{align}
Then we combine \eqref{tauftau} with \eqref{g-lip} to bound the diffusion term as $\|G(X(t_{n-1}))-G(Y_{n-1}^h)\|_{\mathcal L_2^0}^2\le (1+\delta)K_4^2\|e_{n-1}^h\|^2+C \|X(t_{n-1})-{\widehat Y}_{n-1}^h\|^2$ for any $\delta>0$, and then use \eqref{Ftau-}, \eqref{reg-u}, \eqref{reg-aux}, \eqref{err-aux} and H\"older inequality to derive 
\begin{align*}
	\mathbb E\|e^h_n\|^{2p}
	& \le \frac{1}{(1+ 2 (\lambda_1-\varepsilon) \tau)^p}\sum_{r=0}^p \mathcal H_r(\tau)\mathbb E[\|e_{n-1}^h\|^{2r}\|X(t_{n-1})-{\widehat Y}_{n-1}^h\|^{2p-2r}]\\
	& +C\tau (1 + \|X_0\|^{(q+1)(4q+1)}_{1+\gamma} + \|X_0\|_{L_\xi^{2(3q+1)}}^{(q+1)(4q+1)})^{2p} (h^{1+\gamma}+\tau^{1/2})^{2p},
\end{align*}
where
$\mathcal H_r(\tau):=(1+\varepsilon\tau)C_p^r(1+(2K_3+2K_3^2\tau+(2p-1)(1+\delta)K_4^2)\tau)^r(C\tau)^{p-r}$, $r=0,1,\cdots,p$.
In particular,
$\mathcal H_p(\tau)=(1+\varepsilon\tau)(1+(2K_3+2K_3^2\tau+(2p-1)(1+\delta)K_4^2)\tau)^p$.
When $p=1$, the sum $\sum_{r=1}^{p-1}$ is empty and adopted as zero by convention, allowing us to directly obtain the final iterative inequality without further bounding.  
For $p \ge 2$ and $r\in\{1,\cdots,p-1\}$, we apply Young inequality to obtain
\begin{align*}
&\mathcal H_r(\tau)\mathbb E[\|e_{n-1}^h\|^{2r}\|X(t_{n-1})-{\widehat Y}_{n-1}^h\|^{2p-2r}] \\
&\quad\le \frac{\Phi(\tau)}{2(p-1)}\mathbb E\|e_{n-1}^h\|^{2p}
+\frac{p-r}{p}\left(\frac{2r(p-1)}{p\Phi(\tau)}\right)^{\frac{r}{p-r}}\mathcal H_r(\tau)^{\frac{p}{p-r}}\mathbb E\|X(t_{n-1})-{\widehat Y}_{n-1}^h\|^{2p}.
\end{align*}
where $\Phi(\tau):=(1+2(\lambda_1-\varepsilon)\tau)^p-\mathcal H_p(\tau)$.
Summing over $r$ then yields
\begin{align*}
&\sum_{r=1}^{p-1}\mathcal H_r(\tau)\mathbb E[\|e_{n-1}^h\|^{2r}\|X(t_{n-1})-{\widehat Y}_{n-1}^h\|^{2p-2r}] \\
&\quad\le \frac{\Phi(\tau)}{2}\mathbb E\|e_{n-1}^h\|^{2p}
+\sum_{r=1}^{p-1}\frac{p-r}{p}(\frac{2r(p-1)}{p\Phi(\tau)})^{\frac{r}{p-r}}\mathcal H_r(\tau)^{\frac{p}{p-r}}\mathbb E\|X(t_{n-1})-{\widehat Y}_{n-1}^h\|^{2p}.
\end{align*}
Combining the above estimates, we have
\begin{align*}
    \mathbb E\|e^h_n\|^{2p} 
	& \le \Big(1-\frac{\Phi(\tau)}{2(1+2(\lambda_1-\varepsilon)\tau)^p}\Big)\mathbb E\|e_{n-1}^h\|^{2p}+C(\tau)\mathbb E\|X(t_{n-1})-{\widehat Y}_{n-1}^h\|^{2p}\\
	& +C\tau (1 + \|X_0\|^{(q+1)(4q+1)}_{1+\gamma} + \|X_0\|_{L_\xi^{2(3q+1)}}^{(q+1)(4q+1)})^{2p} (h^{1+\gamma}+\tau^{1/2})^{2p},
\end{align*}
where $C(\tau)=(1+2(\lambda_1-\varepsilon)\tau)^{-p} [\mathcal H_0(\tau) + \sum_{r=1}^{p-1}\frac{p-r}{p}(\frac{2r(p-1)}{p\Phi(\tau)})^{\frac{r}{p-r}}\mathcal H_r(\tau)^{\frac{p}{p-r}}]$, with the convention that the sum vanishes if $p=1$.

For notational convenience, set $a_\tau:=1+2(\lambda_1-\varepsilon)\tau$ and $b_\tau:=1+(2K_3+(2p-1)(1+\delta)K_4^2)\tau+2K_3^2\tau^2$.
Using the elementary inequality $x^p-y^p\ge p(x-y)y^{p-1}$ for $x,y>0$ and $p\ge1$, we obtain
\begin{align}\label{Phi-tilde}
\Phi(\tau)&=a_\tau^p-(1+\varepsilon\tau)b_\tau^p \ge b_\tau^{p-1}[p(a_\tau-b_\tau)-\varepsilon\tau b_\tau]\ge c_3\tau,
\end{align}
for some positive constant $c_3$ and for any $\tau\in(0,\tau_3)$, with $\tau_3=\min\{1,(\lambda_1-K_3-\frac{2p-1}{2}K_4^2)K_3^{-2}\}\in(0,1]$ as $\lambda_1>K_3+\frac{2p-1}{2}K_4^2$ and $\varepsilon,\delta>0$ are sufficiently small.
On the other hand, it is clear that $\mathcal H_0(\tau)\le C\tau^p$ and $\mathcal H_r(\tau)\le C\tau^{p-r}$ for $r=1,\cdots,p$.
Combining these with $\Phi(\tau)>c_3\tau$ and
$\Phi(\tau)^{-\frac{r}{p-r}}\mathcal H_r(\tau)^{\frac{p}{p-r}}\le C\tau^{p-\frac{r}{p-r}}\le C\tau$ as $1\le r\le p-1$ when $p\ge2$, we have the existence of a positive constant $c_4$ such that
\begin{align*}
C(\tau)\le c_4\tau.
\end{align*}
Consequently, 
\begin{align*}
    \mathbb E\|e^h_n\|^{2p}
	& \le   (1-c_3\tau) \mathbb E\|e_{n-1}^h\|^{2p} + c_4 \tau \mathbb E\|X(t_{n-1})-{\widehat Y}_{n-1}^h\|^{2p}\\
	& +C\tau (1 + \|X_0\|^{(q+1)(4q+1)}_{1+\gamma} + \|X_0\|_{L_\xi^{2(3q+1)}}^{(q+1)(4q+1)})^{2p} (h^{1+\gamma}+\tau^{1/2})^{2p}.
\end{align*}
Iterating this inequality and using \eqref{err-aux}, we derive
\begin{align*}
\sup_{n\in\mathbb N_+}\|e_n^h\|_{L_\omega^{2p} L_\xi^2}
\le C (1 + \|X_0\|^{(q+1) (4q+1)}_{1+\gamma})
(h^{1+\gamma}+\tau^{1/2}).
\end{align*}
Combining this estimate with \eqref{err-aux} and triangle inequality,
we conclude \eqref{err}.  
\end{proof}

\subsection{Error estimate in additive noise case}

In this part, our main purpose is to establish a uniform-in-time high-order convergence result in the additive noise case, i.e., the diffusion coefficient of Eq. \eqref{see} is a constant operator $G(\cdot) \equiv G$.
 In this case, the conditions \eqref{g-lip}-\eqref{g1} are equivalent to the assumption $G\in \mathcal L_2^1$ which was imposed in \cite{CHS21, QW26} as $\|(-A)^{1/2}G\|_{\mathcal L_2^0}<\infty$.

 \begin{tm}  \label{tm-err-} 
Let $p \ge 1$, $X_0 \in \dot H^{1+\gamma} \cap L_\xi^{2(3q+1)}$ with $\gamma\in [0,1]$, $G(\cdot) \equiv G \in \mathcal L_2^1$ (and $G \in \mathcal L_2^{1+\theta}$ for some $\theta \in (0, 1)$ if $\gamma=1$), and Assumption \ref{ap-f} hold with $\lambda_1>K_3$. 
There exist constants $C>0$ and $\tau_4\in(0,1)$ such that for any $h \in (0,1)$ and $\tau \in (0,\tau_4)$, 
\begin{align} \label{err-} 
& \sup_{n \in \mathbb N_+} \|X(t_n) - Y_n^h \|_{L_\omega^{2p} L_\xi^2} 
\le C (1 +\|X_0\|_{1+\gamma}^{3q^2+6q+1}
+ \|X_0\|_{L_\xi^{2(3q+1)}}^{(q+1)(4q+1)}) (h^{1+\gamma}+\tau^\frac{1+\gamma}2 ).
\end{align} 
 \end{tm}

\begin{proof}
Let $p \ge 2$ and $\gamma \in [0,1]$.
As in the proof of Theorem \ref{tm-err}, following an argument similarly to \eqref{en+}, and utilizing \eqref{Ftau-}, we obtain
\begin{align*}
\|e_n^h  \|^2
& \le C \tau (1+\|X(t_{n-1})\|^q_1+\|\widehat Y_{n-1}^h\|^q_1)^2\|X(t_{n-1})-\widehat Y_{n-1}^h\|^2 \\
& \quad + \frac{1+2\tau(K_3+K_3^2\tau)}{1+2(\lambda_1-\varepsilon) \tau}\|e_{n-1}^h\|^2.
\end{align*}
We note that $\frac{1+2\tau(K_3+K_3^2\tau)}{1+2(\lambda_1-\varepsilon) \tau} < 1$ for any $\tau \in (0, \frac{\lambda_1-K_3}{K_3^2})$, ensured by the condition $\lambda_1> K_3$ for sufficiently small $\varepsilon > 0$.
Consequently, we use \eqref{reg-aux} to derive 
\begin{align}
\sup_{n\in\mathbb N_+}\|e_n^h\|_{L_\omega^{2p}L_\xi^2}
&\le C (1+\|X\|_{L_t^\infty L_\omega^{4pq}\dot H^1}^q
+\sup_{k\in\mathbb N_+}
\|\widehat Y_k^h\|_{L_\omega^{4pq}\dot H^1}^q)
\sup_{k\in\mathbb N_+}
\|X(t_k)-\widehat Y_k^h\|_{L_\omega^{4p}L_\xi^2}.
\label{e-aux-uniform}
\end{align}
Therefore, it suffices to show the following refined estimation of $\sup_{k\in\mathbb N_+}
\|X(t_k)-\widehat Y_k^h\|_{L_\omega^{4p}L_\xi^2}$ with $G(\cdot)\equiv G$.
According to the proof in Lemma \ref{lm-aux}, we only need to give refined estimations for $J_{21}^n$, $J_{23}^n$, and $J_3^n$, respectively.

To show a refined estimation of $J^n  _{21}$, we use Taylor formula to split $J^n  _{21}$ into 
\begin{align*} 
J^n  _{21} 
&\le \Big \|\sum_{i=0}^{n-1} \int_{t_i}^{t_{i+1}} S(t_n-r) 
DF(X(r)) (S(r-t_i)-{\rm Id}) X(t_i) {\rm d}r \Big\|_{L_\omega^p L_\xi^2} \\
&\quad + \Big \|\sum_{i=0}^{n-1} \int_{t_i}^{t_{i+1}} S(t_n-r) 
DF(X(r)) \Big(\int_{t_i}^r S(r-\sigma) F(X(\sigma)) {\rm d}\sigma\Big) {\rm d}r \Big\|_{L_\omega^p L_\xi^2} \\
&\quad + \Big \|\sum_{i=0}^{n-1} \int_{t_i}^{t_{i+1}} S(t_n-r) 
DF(X(r)) \Big(\int_{t_i}^r S(r-\sigma) G {\rm d}W(\sigma) \Big)  {\rm d}r \Big\|_{L_\omega^p L_\xi^2} \\
&\quad +\Big \|\sum_{i=0}^{n-1} \int_{t_i}^{t_{i+1}} S(t_n-r) 
R_F(X(r), X(t_i)) {\rm d}r \Big\|_{L_\omega^p L_\xi^2}
=:\sum_{j=1}^4 J^n  _{21j}, 
\end{align*}
where $R_F$ denotes the remainder term
\begin{align*} 
& R_F(X(r), X(t_i)) \\
&:=\int_0^1 D^2 F (X(r)+\lambda (X(t_i)-X(r)) )
 (X(t_i)-X(r), X(t_i)-X(r)) (1-\lambda) {\rm d}\lambda \\
&=\int_0^1 f'' (X(r)+\lambda (X(t_i)-X(r)) )
 (X(t_i)-X(r))^2 (1-\lambda) {\rm d}\lambda.
\end{align*}
We shall estimate $J^n  _{21i}$, $i=1,2,3,4$, successively.

Since for $\delta\in (3/2, 2)$, $\dot H^\delta\hookrightarrow \mathcal C$, it follows by the dual argument that
\begin{align} \label{l1} 
\|x\|_{-\delta}\le C \|x\|_{L_\xi^1},
\quad x\in L_\xi^1.
\end{align}
This inequality, in conjunction with Minkovski and Cauchy--Schwarz inequalities, \eqref{f-grow}, \eqref{emb}, \eqref{ana} with $(\mu, \nu)=(\delta, 0)$, and \eqref{ana-hol} with $\rho=1+\gamma$, yields that 
\begin{align} \label{j211} 
J^n  _{211} 
 &\le \sum_{i=0}^{n-1} \int_{t_i}^{t_{i+1}} (t_n-r)^{-\frac\delta2} \|DF(X(r)) (S(r-t_i)-{\rm Id}) X(t_i)\|_{L_\omega^p \dot H^{-\delta}} {\rm d}r \nonumber \\
&\le C \tau^\frac{1+\gamma}2(1+\|X\|_{L_t^\infty L_\omega^{2pq} \dot H^1}^q) \|X\|_{L_t^\infty L_\omega^{2pq} \dot H^{1+\gamma}},
\quad \gamma \in [0, 1].
\end{align}
A similar argument implies that 
\begin{align} \label{j212} 
J^n  _{212} 
&\le C \tau (1+\|X\|_{L_t^\infty L_\omega^{2pq} \dot H^1}^{q+1}).
\end{align}

To estimate the third term $J^n  _{213}$, we apply stochastic Fubini theorem, the discrete and continuous BDG inequalities, and Cauchy--Schwarz inequality to derive 
\begin{align*}
(J^n  _{213})^2 
&=\Big\| \sum_{i=0}^{n-1} \int_{t_i}^{t_{i+1}} \int_{t_i}^{t_{i+1}} 
\chi_{[t_i,r)}(\sigma) S(t_n-r) 
DF(X(r)) S(r-\sigma) G {\rm d}r {\rm d}W(\sigma) \Big\|_{L_\omega^p L_\xi^2}^2 \\ 
&\le C \tau \sum_{i=0}^{n-1} 
\Big(\int_{t_i}^{t_{i+1}} \|f'(X(r))\|^2_{L_\omega^p \dot H^{-1}}  {\rm d}r \Big)
\Big(\int_0^\tau  \|S(\sigma) G\|_{\mathcal L_2^1}^2  {\rm d}\sigma \Big).
\end{align*}
This, together with the condition \eqref{f-grow} and the assumption $G\in \mathcal L_2^1$, shows that 
\begin{align} \label{j213}
J^n  _{213} \le C \tau (1+\|X\|_{L_t^\infty L_\omega^{pq} \dot H^1}^q).
\end{align}

Finally, by \eqref{ana} with $(\mu, \nu)=(\delta, 0)$, Minkovski and H\"older inequalities, the condition \eqref{f-grow}, and the embeddings \eqref{emb}, \eqref{l1}, $\dot H^{1/2} \hookrightarrow L_\xi^3$, and $\dot H^1 \subset L_\xi^{3(q-1)}$ as $3(q-1)\le 2(q+1)$ (for $q \le 2$) for $d=3$ and $\dot H^1 \subset L_\xi^p$ for any $p \in [1, \infty)$ for $d=1, 2$, we estimate the term $J^n  _{214}$ by
\begin{align} \label{j214}
J^n  _{214} 
&\le C \sup_{i \in \mathbb N}  \int_0^1 \|f''(X+\lambda (X(t_i)-X) ) \|_{L_\omega^{3p} L_\xi^3} \|X(t_i)-X \|^2_{L_\omega^{3p} L_\xi^3}{\rm d}\lambda \nonumber \\
& \le C \tau^{\min\{1/2+\gamma, 1\}}(1+\|X\|^{q-1}_{L_t^\infty L_\omega^{3p(q-1)} \dot H^1} )\|X\|^2_{\mathcal C_t^{\min\{1/2+\gamma, 1\}/2} L_\omega^{3p} \dot H^{1/2}}.
\end{align} 

Collecting the above four estimations \eqref{j211}-\eqref{j214} together, 
we have 
\begin{align} \label{j21}
J^n  _{21} 
& \le C \tau^\frac{1+\gamma}2 [1+\|X\|_{L_t^\infty L_\omega^{2pq} \dot H^1}^{q+1}
+ (1+\|X\|_{L_t^\infty L_\omega^{2pq} \dot H^1}^q) \|X\|_{L_t^\infty L_\omega^{2pq} \dot H^{1+\gamma}} \nonumber  \\
& \qquad \qquad + (1+\|X\|^{q-1}_{L_t^\infty L_\omega^{3p(q-1)} \dot H^1} )\|X\|^2_{\mathcal C_t^{\min\{1/2+\gamma, 1\}/2} L_\omega^{3p} \dot H^{1/2}}] \nonumber \\
& \le C \tau^\frac{1+\gamma}2 (1+\|X_0\|^{2q^2+5q+1}_{1+\gamma}). 
\end{align} 

For the term $J^n_{23}$, we use the estimate \eqref{F-Ftau+}, instead of \eqref{F-Ftau}, \eqref{sht-1}, \eqref{reg-u+}, and the elementary estimate $\sum_{i=1}^\infty t_i^{-1/2} (1+ c \tau)^{-i}  \le C \tau^{-1}$ to derive  
\begin{align} \label{j23-} 
J^n_{23} 
& \le C \sum_{i=0}^{n-1} \int_{t_i}^{t_{i+1}} t_{n-i}^{-1/2} \Big(1+ \frac12\tau \lambda^* \Big)^{- (n-i)}
\|F(X(t_i))-F_\tau(X(t_i))\|_{L_\omega^p L_\xi^2} {\rm d}r \nonumber  \\
&\le C \tau (1+\|X\|_{L_t^\infty L_\omega^{p(3q+1)} L_\xi^{2(3q+1)}}^{3q+1}).
\end{align}   
The last term $J^n_3$ can be handled by BDG inequality and \eqref{eht2} with $\theta=1$:
\begin{align}  \label{j3+} 
J^n_3 = J^n_{32} = \Big\|\int_0^{t_n} 
E_{h,\tau}(t_n-r) G {\rm d}W(r) \Big \|_{L_\omega^p L_\xi^2} 
 \le C\|G \|_{\mathcal L_2^1} (h^2+\tau).
\end{align}
Now we can conclude \eqref{err-} from the error estimates \eqref{j1}, \eqref{j21}, \eqref{j22}, \eqref{j22+},  \eqref{j23}, and \eqref{j3+} for $J^n_1$, $J^n  _{21}$, $J^n_{22}$, $J^n_{23}$, and $J^n_3$, respectively, and \eqref{reg-u}, \eqref{reg-u+}, and \eqref{reg-aux}.
\end{proof}

\section{Exponential Ergodicity of Tamed-FEM}
\label{sec5}

In this section, we establish exponential ergodicity of the tamed-FEM \eqref{t-fem} and derive an ergodic error estimate in the Wasserstein-2 distance.
 
It was shown in \cite[Theorem 3.3]{LS26} that, under the coercivity condition \eqref{f'} and the first growth condition in \eqref{f-grow} on $f$ and the Lipschitz condition \eqref{g-lip} on $g$, 
there exist positive constants $c_5$, $c_6$, and $\tau_5 \le 1$ such that for any $h \in (0, 1)$ and $\tau \in (0, \tau_5]$,  
  \begin{align} \label{lya} 
 \mathbb E_{n-1} \|Y_n^h\|^2  
  \le  (1- c_5 \tau) \|Y_{n-1}^h\|^2+ c_6 \tau, \quad n \in \mathbb N_+. 
  \end{align}    
Therefore, to obtain the exponential ergodicity of the tamed-FEM \eqref{t-fem}, it suffices to show the following exponential continuity estimate of \eqref{t-fem} with respect to the initial data as in \cite[Theorem 3.1]{Liu26}; here we denote by $\{Y_m^{h,x}: m \in \mathbb N\}$ and $\{Y_m^{h,y}: m \in \mathbb N\}$ the solutions of the tamed-FEM \eqref{t-fem} starting from $x$ and $y$, respectively.

\begin{lm}  
Let $x, y \in H$ and Assumptions \ref{ap-f}-\ref{ap-g}(1) hold with $\lambda_1>K_3+\frac{1}{2}K_4^2$. 
There exist positive constants $\kappa$ and $\tau_6<1$ such that for any $h \in (0, 1)$ and $\tau \in (0, \tau_6)$,
\begin{align} \label{con-ukh}
\mathbb E \|Y_m^{h,x}-Y_m^{h,y}\|^2 \le e^{-\kappa t_m} \|x-y\|^2, 
\quad m \in \mathbb N.
\end{align}  
\end{lm}

\begin{proof}

For $k \in \mathbb N$, set $F_k^{h,x}=F_\tau(Y_k^{h,x})$, $F_k^{h,y}=F_\tau(Y_k^{h,y})$, $G_k^{h,x}=G(Y_k^{h,x})$, $G_k^{h,y}=G(Y_k^{h,y})$, and $E_k^h=Y_k^{h,x}-Y_k^{h,y}$. 
Using \eqref{ab}, we have  
\begin{align*} 
& \|E_{k+1}^h\|^2 - \|E_k^h\|^2 +  \|E_{k+1}^h-E_k^h\|^2 + 2\tau\|\nabla E_{k+1}^h\|^2 \\
& = 2\tau \langle F^{h,x}_k-F^{h,y}_k, E_k^h\rangle_{-1, 1} + 2 \langle E_k^h, (G^{h,x}_k-G^{h,y}_k) \delta_k W\rangle \\
& \quad + 2 \langle E_{k+1}^h-E_k^h, \tau(F^{h,x}_k-F^{h,y}_k)+(G^{h,x}_k-G^{h,y}_k) \delta_k W\rangle .
\end{align*}   
It follows from \eqref{poin}, \eqref{tauftau}, Cauchy--Schwarz inequality, It\^o isometry, the fact that $G^{h,x}_k-G^{h,y}_k$ is independent of  $\delta_k W$, and the condition \eqref{g-lip} that  
\begin{align*}  
& (1+2\lambda_1\tau) \mathbb E \|E_{k+1}^h\|^2
\le (1+\tau(2K_3+2K_3^2\tau+K_4^2)) \mathbb E \|E_k^h\|^2,
\end{align*}   
from which we obtain
\begin{align*}  
\mathbb E \|E_k^h\|^2 
\le \Big(\frac{1+\tau(2K_3+2K_3^2\tau+K_4^2)}{1+2\lambda_1 \tau}\Big)^k \|x-y\|^2, 
\quad k \in \mathbb N.
\end{align*}   
Note that $a^k\le e^{-(1-a)k}$ for $a\in(0,1)$, and that $\frac{1+\tau(2K_3+2K_3^2\tau+K_4^2)}{1+2\lambda_1 \tau}<1$ for any $\tau \in (0, \tau_6)$, with $\tau_6:=\min\{1, (\lambda_1-K_3-\frac{1}{2}K_4^2) K_3^{-2}\} \in (0, 1]$ which is ensured by the condition $\lambda_1>K_3+\frac{1}{2}K_4^2$, we conclude \eqref{con-ukh}.
\end{proof}

By combining the Lyapunov structure \eqref{lya}, the continuity estimate \eqref{con-ukh}, and the uniform-in-time strong error estimates from Theorems \ref{tm-err} and \ref{tm-err-}, we establish the exponential ergodicity of the tamed-FEM \eqref{t-fem} with quantitative estimates in the Wasserstein-$2$ distance.

\begin{tm} \label{tm-erg}
Let Assumptions \ref{ap-f}-\ref{ap-g}(1) hold with $\lambda_1>K_3+\frac12\max\{ K_4^2, (2(3q\\+1)(q+1)-1)K_6\}$. 
There exists a positive constant $\tau_{\max}<1$ such that for any $\tau\in(0,\tau_{\max})$ and $h\in(0,1)$, the tamed-FEM \eqref{t-fem} possesses a unique invariant measure $\pi_\tau^h$ on $V_h$ and is exponentially mixing.
Moreover, for any $\varepsilon\in(0,1)$, there exists a constant $C$ such that $\mathbb W_2(\pi_0, \pi_\tau^h) \le C (h^{2-\varepsilon}+\tau^{1/2})$, where $\pi_0$ denotes the unique invariant measure of Eq. \eqref{see}.
In particular, for the additive noise case with $G(\cdot) \equiv G \in \mathcal L_2^1$, if Assumption \ref{ap-f} holds with $\lambda_1>K_3$, then $\mathbb W_2(\pi_0, \pi_\tau^h) \le C(h^{2-\varepsilon}+\tau^{1-\varepsilon})$.
\end{tm}

\section{Numerical Experiments}
\label{sec6}

In this section, we present numerical tests to illustrate the ergodic behavior, as well as the sharpness and time-independence of the strong convergence rates in Theorems \ref{tm-err} and \ref{tm-err-} for the tamed-FEM \eqref{t-fem}, applied to the stochastic Allen--Cahn equation, i.e., Eq. \eqref{see-fg} with $f(\xi)=\epsilon^{-2}(\xi-\xi^3)$, $\xi \in \mathbb R$. 
 
All experiments are performed on the two-dimensional unit square $\mathcal O=(0,1)^2$ with homogeneous DBC. To test both additive and multiplicative noise, we consider the trace-class noise coefficient $g(\xi)=\sigma_1 + \sigma_2 \xi$, $\xi \in \mathbb R$. We set $\varepsilon=0.25$ and choose $(\sigma_1, \sigma_2)=(2, 0)$ for the additive noise case, and $(\sigma_1, \sigma_2)=(2, 0.8)$ for the multiplicative noise case. 
  
The noise is expanded in the orthonormal sine basis $\{g_{k,\ell}\}$ with corresponding eigenvalues $\{q_{k,\ell}\}$:
\begin{align*}
    g_{k,\ell}(\xi_1,\xi_2)=2\sin(k\pi\xi_1)\sin(\ell\pi\xi_2), \quad q_{k,\ell}=(1+k^2+\ell^2)^{-3}, 
    \qquad k,~ \ell \in \mathbb N_+.
\end{align*} 
In the simulations, we truncate the $Q$-Wiener process $W$ and obtain the truncated Wiener process
$W^N =\sum_{1\le k,\ell\le N}(1+k^2+\ell^2)^{-3/2} g_{k,\ell} \beta_{k,\ell}$, where $\{\beta_{k,\ell}\}$ is a family of independent 1D Brownian motions.

Since the functions $g_{k,\ell}$ are also Dirichlet eigenfunctions on $(0,1)^2$, we use the same notation when defining the initial data and the observables below. The temporal tests use $N=32$, whereas the spatial and ergodicity tests use $N=64$. 

\subsection{Uniform-in-time strong convergence tests}

To construct initial data with varying regularities, we define the following for $M\in\mathbb N_+$ and $\alpha,p>0$: 
\begin{align*}
	u_{\alpha,p}^{M}(\xi)=\sum_{k,\ell=1}^{M}\frac{\eta_{k,\ell}g_{k,\ell}(\xi)}{(k^2+\ell^2)^{\alpha}\log^{p}(e+\sqrt{k^2+\ell^2})},
	\qquad
	\bar u_{\alpha,p}^{M}=\frac{u_{\alpha,p}^{M}}{\|u_{\alpha,p}^{M}\|_{L^2(\mathcal O)}},
\end{align*}
where $\eta_{k,\ell}\in\{-1,1\}$ are fixed signs. Slower spectral decay yields a rougher profile, whereas faster decay produces a smoother one. Accordingly, we take
\begin{align*}
	X_{0,\mathrm{time}}^{(1)}=0.8\,g_{1,1}+\bar u_{1,1.2}^{64},
	\qquad
	X_{0,\mathrm{time}}^{(2)}=0.8\,g_{1,1}+\bar u_{1.5,2}^{64},
\end{align*}
for the temporal tests, and
\begin{align*}
	X_{0,\mathrm{space}}^{(1)}=0.8\,g_{1,1}+\bar u_{1,0.51}^{512},
	\qquad
	X_{0,\mathrm{space}}^{(2)}=0.8\,g_{1,1}+\bar u_{1.5,1}^{512},
\end{align*}
for the spatial tests. 
Although these series are truncated, the two choices are designed to mimic data near the $\dot H^1$ and $\dot H^2$ regularity thresholds, respectively.

{\bf Temporal strong convergence.}
We fix $h=1/32$ and $T=50$. The reference solution is computed on the same finite element space with $\tau_{\rm ref}=2^{-14}$ for additive noise and $\tau_{\rm ref}=2^{-13}$ for multiplicative noise. 
In each Monte--Carlo sample, the coarse and reference paths are coupled by the same truncated Wiener path, with coarse increments obtained by summing the corresponding fine increments. For $I\subset[0,T]$, define
\begin{align*}
    E_I(\tau)
    = \max_{t_j\in I}
      \Big(
      \frac1{K_{\rm mc}}\sum_{m=1}^{K_{\rm mc}}
       \|U_j^{h,\tau,(m)}-U_{\rm ref}^{h,\tau_{\rm ref},(m)}(t_j) \|_{L^2(\mathcal O)}^2
      \Big)^{1/2}.
\end{align*}
We also compute the normalized constant
\begin{align*}
    C_{[0,T]}^{\rm time}
    = \max_{\tau}
    \frac{E_{[0,T]}(\tau)}{\tau^\beta},
    \qquad
    \beta=\frac12 \ \hbox{for } X_0^{(1)},\quad
    \beta=1 \ \hbox{for } X_0^{(2)}.
\end{align*}
Thus, the maximum over time is taken after forming the Monte--Carlo mean-square error at each sampled time. With $K_{\rm mc}=500$, Fig. \ref{fig:t-add} gives fitted rates $0.560$ and $0.934$ for the low- and high-regularity data in the additive case, which are consistent with the expected orders $\tau^{1/2}$ and $\tau$ Fig. \ref{fig:t-mul} gives rate $0.576$ for the multiplicative case with $X_0^{(1)}$. In all cases, $C_{[0,T]}^{\rm time}$ remains stable as $T$ grows from $1$ to $50$, illustrating the uniform-in-time behavior.

\begin{figure}[htbp]
\centering
\includegraphics[width=\textwidth]{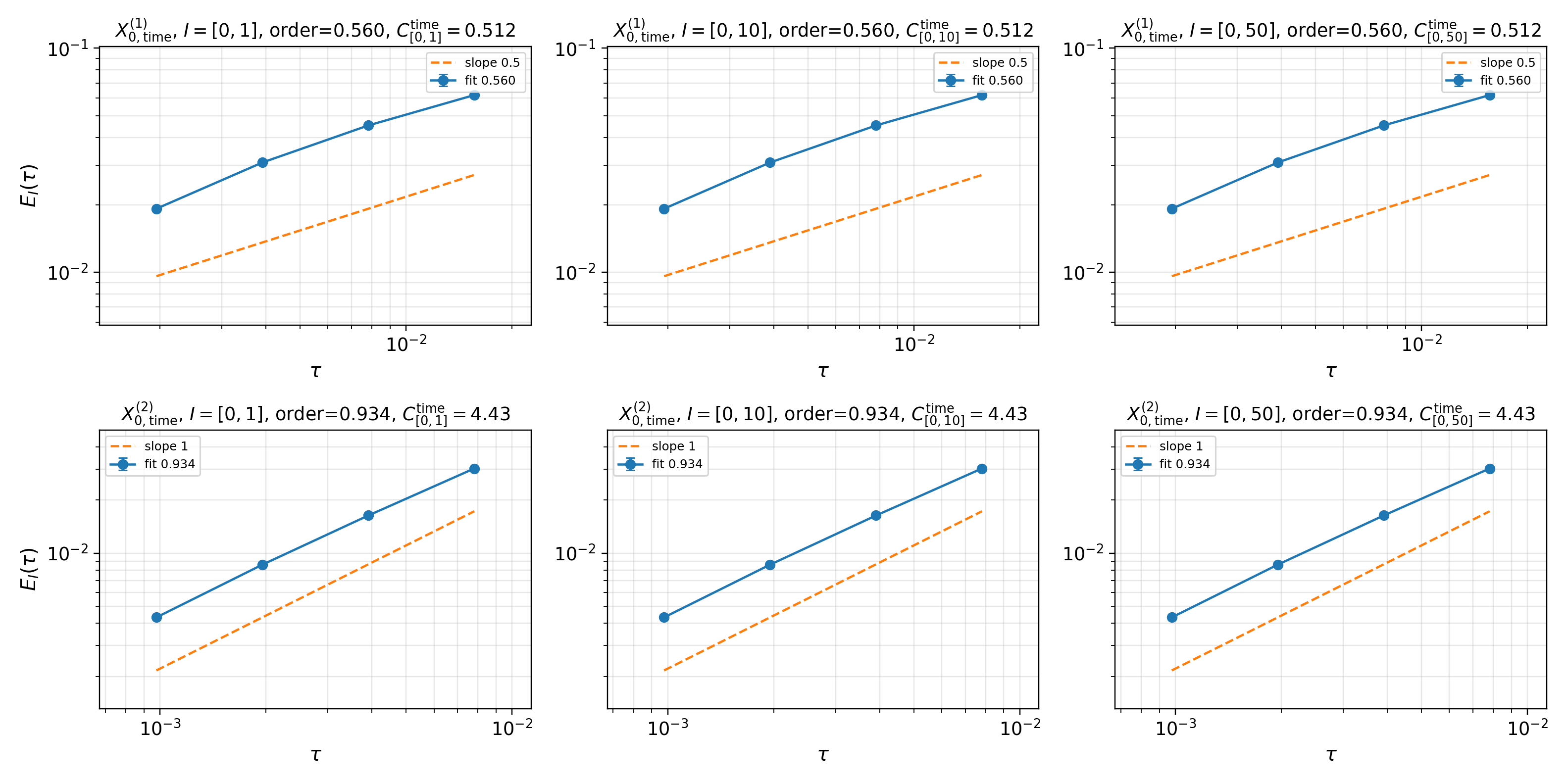}
\caption{Temporal uniform-in-time orders for additive noise.}
\label{fig:t-add}
\end{figure}

\begin{figure}[htbp]
\centering
\includegraphics[width=\textwidth]{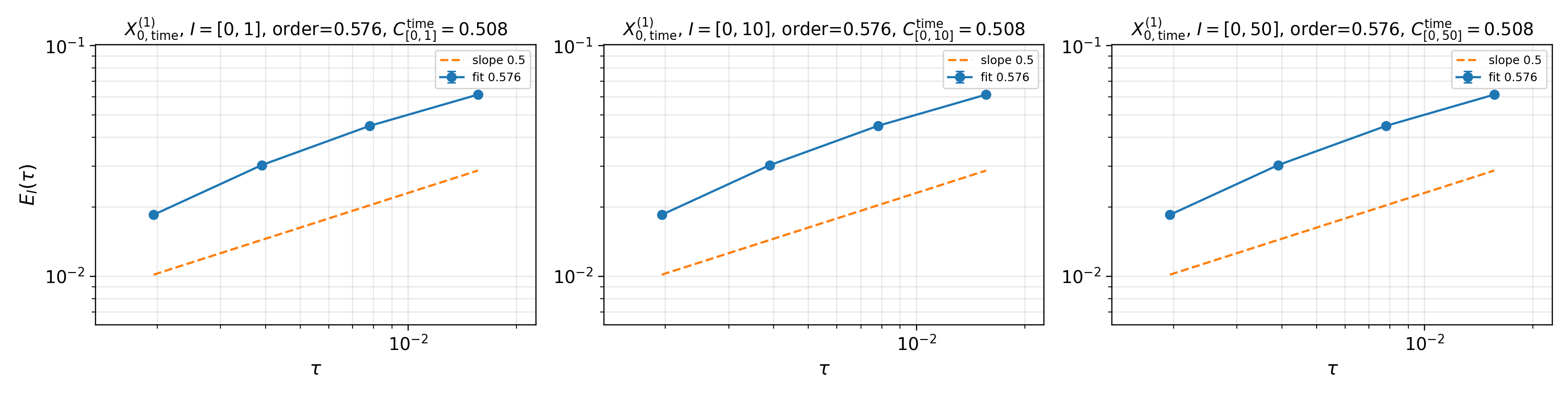}
\caption{Temporal uniform-in-time orders for multiplicative noise.}
\label{fig:t-mul}
\end{figure}

{\bf Spatial strong convergence.}
We fix $\tau=2^{-8}$, $T=50$, and compute the reference solution on $h_{\rm ref}=1/256$. The tested meshes are $h=1/n$, $n\in\{8,16,32,64\}$. In each of the $K_{\rm mc}=100$ samples, the coarse and reference solutions use the same truncated Wiener path. The coarse solution is prolonged to the reference mesh by the piecewise linear interpolation $I_h^{h_{\rm ref}}$. For $I\subset[0,T]$, set
\begin{align*}
    E_I(h)
    = \max_{t_j\in I}\Big(
      \frac1{K_{\rm mc}}\sum_{m=1}^{K_{\rm mc}}
      \|I_h^{h_{\rm ref}}U_j^{h,\tau,(m)}-U_j^{h_{\rm ref},\tau,(m)} \|_{L^2(\mathcal O)}^2
      \Big)^{1/2}.
\end{align*}
Thus, the maximum time over $I$ is taken after forming the Monte--Carlo mean-square error at each sampled time. We also compute
\begin{align*}
    C_{[0,T]}^{\rm space}
    = \max_h \frac{E_{[0,T]}(h)}{h^\beta},
    \qquad
    \beta=1 \ \hbox{for } X_0^{(1)},\quad
    \beta=2 \ \hbox{for } X_0^{(2)}.
\end{align*}
Figs. \ref{fig:s-add} and \ref{fig:s-mul} show the fitted spatial strong orders of $1.110$ and $1.941$ for the low- and high-regularity profiles, respectively, which are consistent with the rates predicted
by Theorems \ref{tm-err} and \ref{tm-err-}. The constants $C_{[0,T]}^{\rm space}$ are again essentially unchanged when $T$ grows from $1$ to $50$.

\begin{figure}[htbp]
\centering
\includegraphics[width=\textwidth]{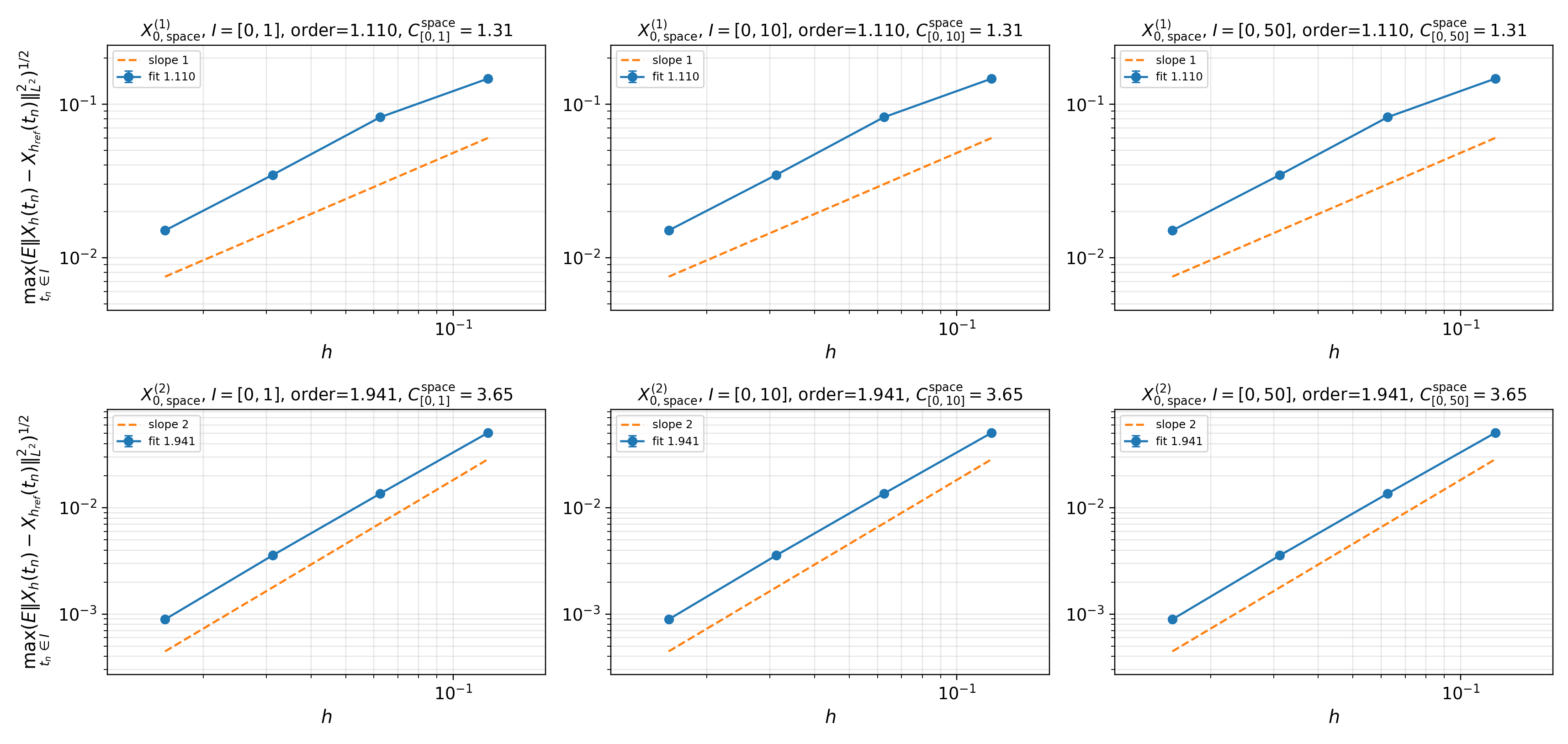}
\caption{Spatial uniform-in-time orders for additive noise.}
\label{fig:s-add}
\end{figure}

\begin{figure}[htbp]
\centering
\includegraphics[width=\textwidth]{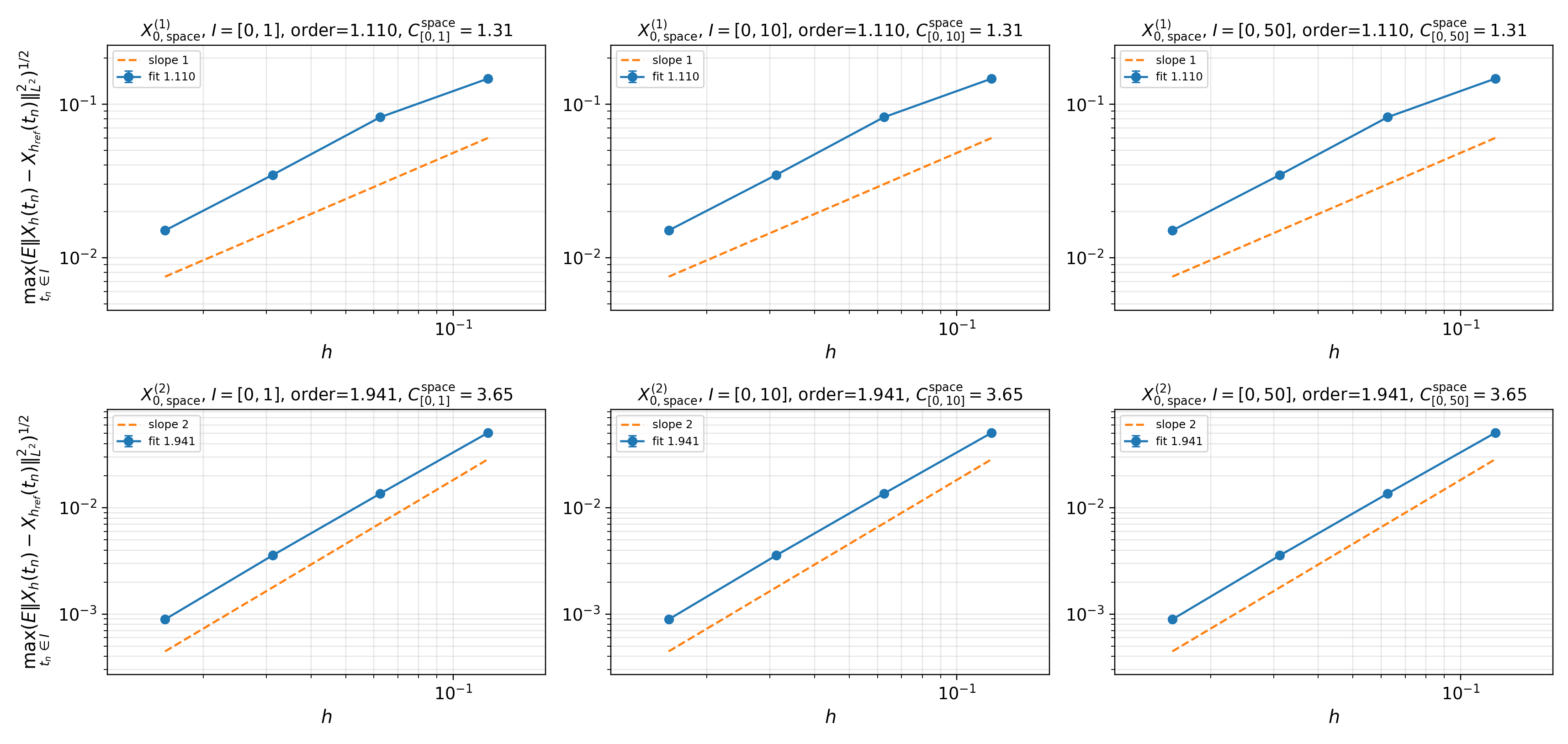}
\caption{Spatial uniform-in-time orders for multiplicative noise.}
\label{fig:s-mul}
\end{figure}

\subsection{Ergodicity test}
We take $h=1/32$, $\tau=2^{-8}$, $T=50$, and $K_{\rm mc}=1000$. The three initial states are
\begin{align*}
	X_0^{(1)}=1.4g_{1,1},\qquad
    X_0^{(2)}=-1.4g_{1,1},\qquad
    X_0^{(3)}=\frac{g_{2,1}+g_{1,2}}{\sqrt2}.
\end{align*}
For each Monte--Carlo path, all three initial states are driven by the same truncated Wiener path. We monitor
\begin{align*}
\varphi_1(u)&=\tanh(\langle  u,g_{1,1}\rangle ),\\
\varphi_2(u)&=\tanh\big(\langle  u,g_{2,1}+g_{1,2}\rangle /\sqrt2\big),\\
\varphi_3(u)&=1-\exp\big(-\|u\|_{L^2(\mathcal O)}^2\big).
\end{align*}
Fig. \ref{fig:ergodic} shows that the ensemble means from the three initial states rapidly merge in both the additive and multiplicative cases, illustrating the loss of memory of initial data and supporting the exponential ergodicity proved in Section \ref{sec5}.

\begin{figure}[htbp]
\centering
\includegraphics[width=\textwidth]{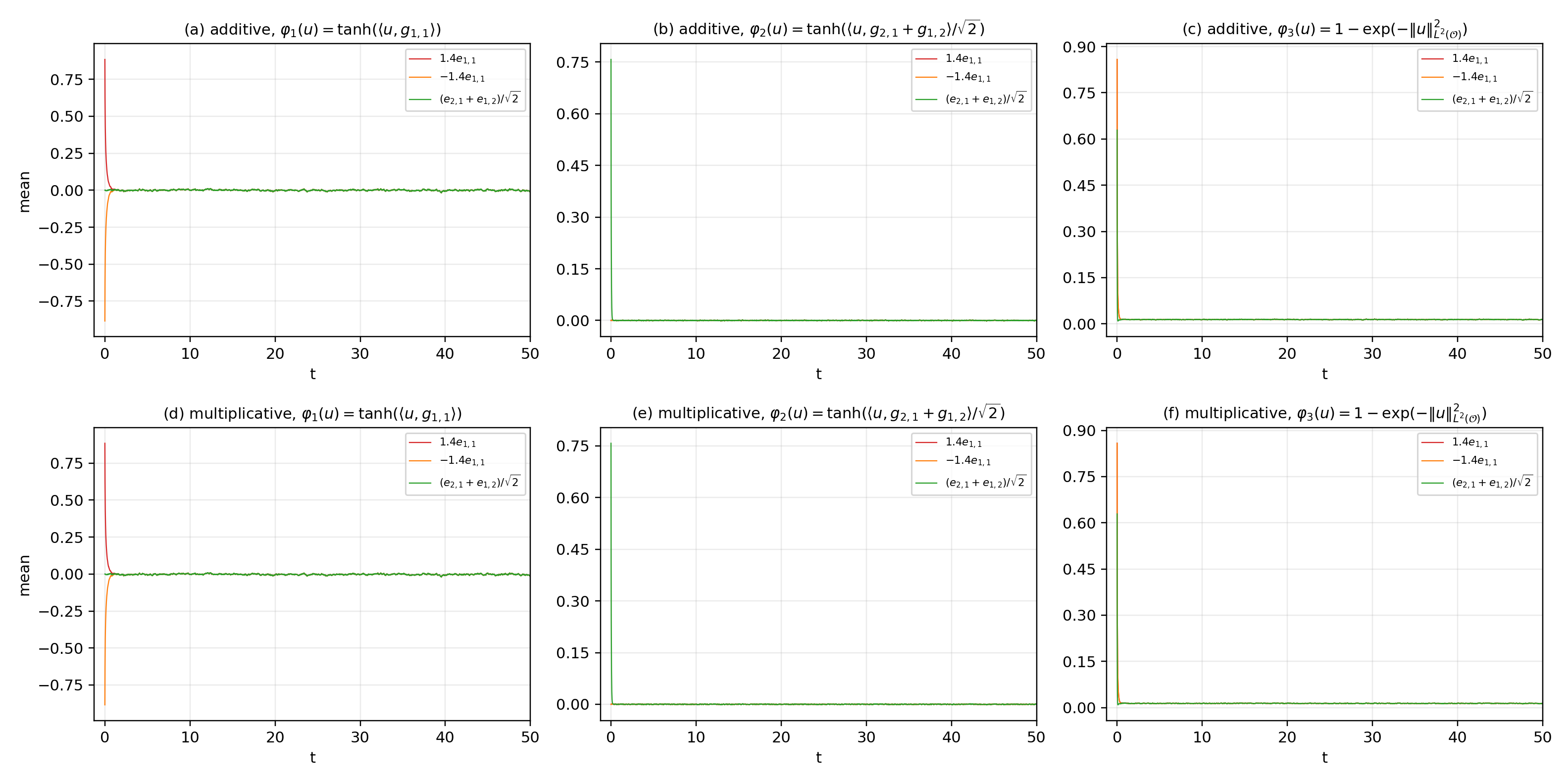}
\caption{Ergodicity tests for additive noise (top) and multiplicative noise (bottom).}
\label{fig:ergodic}
\end{figure}

\bibliographystyle{plain}
\bibliography{references.bib}
\end{document}